\pdfoutput=1
\documentclass{article}
\usepackage{amsmath, amsthm, amssymb, amstext, amsfonts}
\usepackage{enumerate}
\usepackage[applemac]{inputenc}
\usepackage{graphicx}
\usepackage[format=hang, font=small, labelfont=it, margin=6mm]{caption}

\theoremstyle{plain}
\newtheorem{thm}{Theorem}[section]
\newtheorem{lem}[thm]{Lemma}
\newtheorem{cor}[thm]{Corollary}
\newtheorem{prop}[thm]{Proposition}

\newtheorem*{thm1}{Theorem 1}
\newtheorem*{thm2}{Theorem 2}

\theoremstyle{definition}

\newtheorem{ex}{Example}

\newcount\commentno
\def\COMMENT#1{$^{<\the\commentno>}$%
     \vadjust{\vbox to 0pt{\vss\vskip-8pt\rightline{%
     \rlap{\hbox{\hskip7mm \vbox{\pretolerance=-1
     \doublehyphendemerits=0 \finalhyphendemerits=0
     \hsize40mm\tolerance=10000\eightpoint
     \lineskip=0pt\lineskiplimit=0pt
     \rightskip=0pt plus16mm\baselineskip8pt\noindent
     \hskip0pt       
     {$\langle$\the\commentno. #1$\rangle$}\endgraf}}}}\vss}}%
     \global\advance\commentno by1}%
\def\writecommentsasfootnotes{%
 \def\COMMENT{\global\advance\commentno by1\footnote{$^{<\the\commentno>}$}}%
 }
\def\nocomments{\def\COMMENT##1{}}
\def\?#1{\vadjust{\vbox to 0pt{\vss\vskip-8pt\leftline{%
     \llap{\hbox{\vbox{\pretolerance=-1
     \doublehyphendemerits=0\finalhyphendemerits=0
     \hsize16truemm\tolerance=10000\small
     \lineskip=0pt\lineskiplimit=0pt
     \rightskip=0pt plus16truemm\baselineskip8pt\noindent
     \hskip0pt        
     #1\endgraf}\hskip7truemm}}}\vss}}}
\newenvironment{txteq}
  {
    \begin{equation}
    \begin{minipage}[c]{0.85\textwidth} 
    \em                                
  }
  {\end{minipage}\end{equation}\ignorespacesafterend}
\newenvironment{txteq*}
  {
    \begin{equation*}
    \begin{minipage}[c]{0.85\textwidth} 
    \em                                
  }
  {\end{minipage}\end{equation*}\ignorespacesafterend}
\def\specrel#1#2{\mathrel{\mathop{\kern0pt #1}\limits_{#2}}}
\def\Specrel#1#2{\mathrel{\mathop{\kern0pt #1}\limits^{#2}}}

\newcommand{\N}{\ensuremath{\mathbb{N}}}

\newcommand{\Lra}{\ensuremath{\Leftrightarrow}}

\newcommand{\sm}{\ensuremath{\smallsetminus}}

\newcommand{\Aut}{\textnormal{Aut}}
\newcommand{\closure}[1]{\overline{#1}}

\newcommand{\es}{\ensuremath{\emptyset}}

\newcommand{\sub}{\subseteq}

\newcommand{\cB}{\ensuremath{\mathcal B}}

\newcommand{\cE}{\ensuremath{\mathcal E}}

\newcommand{\cI}{\ensuremath{\mathcal I}}

\newcommand{\cN}{\ensuremath{\mathcal N}}

\newcommand{\cP}{\ensuremath{\mathcal P}}

\newcommand{\cR}{\ensuremath{\mathcal R}}
\newcommand{\cS}{\ensuremath{\mathcal S}}
\newcommand{\cT}{\ensuremath{\mathcal T}}

\newcommand{\cV}{\ensuremath{\mathcal V}}

\newcommand{\cX}{\ensuremath{\mathcal X}}

\def\td{tree-decom\-po\-si\-tion}

\newcommand{\csepn}[4]{\ensuremath{{(#1 \cap #2,}\, {#3 \cup #4)}}}

\newcommand{\sys}{separation system}

\newcommand{\sep}[2]{\ensuremath{{#1 \cap #2}}}
\newcommand{\sepn}[2]{\ensuremath{{(#1,#2)}}}
\newcommand{\AB}{\sepn AB}
\newcommand{\BA}{\sepn BA}
\newcommand{\CD}{\sepn CD}
\newcommand{\DC}{\sepn DC}
\newcommand{\EF}{\sepn EF}
\newcommand{\FE}{\sepn FE}
\newcommand{\XY}{\sepn XY}

\newcommand{\pre}{\ensuremath{\mbox{ is a predecessor of }}}
\newcommand{\minor}{\ensuremath{\preccurlyeq}}

\newcommand{\T}{\ensuremath{\cT(\cN)}}
\renewcommand{\cX}{\ensuremath{t}}

\newcommand{\alt}[2]{\left\{\begin{array}{l}#1\\#2\end{array} \right.}
\newcommand{\claim}[2]{\begin{enumerate}\item[#1] {\it #2} \end{enumerate}}

\renewcommand{\diamond}{{*}}

\let\|\parallel

\nocomments

\title{Connectivity and tree structure\\ in finite graphs}
\date{20 March, 2013}

\author{J. Carmesin \and R. Diestel \and F. Hundertmark \and M. Stein\footnote{Supported by Fondecyt grant 11090141.}}

\begin{document}

\maketitle

\begin{abstract}\noindent
Considering systems of separations in a graph that separate every pair of a given set of vertex sets that are themselves not separated by these separations, we determine conditions under which such a separation system contains a nested subsystem that still separates those sets and is invariant under the automorphisms of the graph.

As an application, we show that the $k$-blocks~-- the maximal vertex sets that cannot be separated by at most $k$ vertices~-- of a graph~$G$ live in distinct parts of a suitable \td\ of~$G$ of adhesion at most~$k$, whose decomposition tree is invariant under the automorphisms of~$G$. This extends recent work of Dunwoody and Kr\"on and, like theirs, generalizes a similar theorem of Tutte for $k=2$.

Under mild additional assumptions, which are necessary, our decompositions can be combined into one overall \td\ that distinguishes, for all $k$ simultaneously, all the $k$-blocks of a finite graph.

\end{abstract}

\section{Introduction} \label{sec_intro}

Ever since graph connectivity began to be systematically studied, from about 1960 onwards, it has been an equally attractive and elusive quest to `decompose a $k$-connected graph into its $(k+1)$-connected components'. The idea was modelled on the well-known block-cutvertex tree, which for $k=1$ displays the global structure of a connected graph `up to 2-connectedness'. For general~$k$, the precise meaning of what those `$(k+1)$-connected components' should be varied, and came to be considered as part of the problem. But the aim was clear: it should allow for a decomposition of the graph into those `components', so that their relative structure would display some more global structure of the graph.\looseness=-1

While originally, perhaps, these `components' were thought of as subgraphs, it soon became clear that, for larger~$k$, they would have to be defined differently.
For $k=2$, Tutte~\cite{TutteGrTh} found a decomposition which, in modern terms,%
   \footnote{Readers not acquainted with the terminology of graph minor theory can skip the details of this example without loss. The main point is that those `torsos' are not subgraphs, but subgraphs plus some additional edges reflecting the additional connectivity that the rest of the graph provides for their vertices.}
   would be described as a \td\ of adhesion~2 whose torsos are either 3-connected or cycles.

For general~$k$, Robertson and Seymour~\cite{GMX} re-interpreted those `$(k+1)$-connected components' in a radically new (but interesting) way as `tangles of order~$k+1$'. They showed, as a cornerstone of their theory on graph minors, that every finite graph admits a \td\ that separates all its maximal tangles, regardless of their order, in that they inhabit different parts of the decomposition. Note that this solves the modified problem for all $k$ simultaneously, a feature we shall achieve also for the original problem.

More recently still, Dunwoody and Kr\"on~\cite{DunwoodyKroenArXiv}, taking their lead directly from Tutte (and from Dunwoody's earlier work on tree-structure induced by edge-cuts~\cite{CuttingUpGraphs}), followed up Tutte's observation that his result for $k=2$ can alternatively be described as a tree-like decomposition of a graph $G$ into cycles and vertex sets that are `2-inseparable': such that no set of at most~2 vertices can separate any two vertices of that set in~$G$. Note that such `$k$-inseparable' sets of vertices, which were first studied by Mader~\cite{mader78}, differ markedly from $k$-connected subgraphs, in that their connectivity resides not on the set itself but in the ambient graph. For example, joining $r>k$ isolated vertices pairwise by $k+1$ independent paths of length~2, all disjoint, makes this set into a `$k$-block', a maximal $k$-inseparable set of vertices.%
   \COMMENT{}
This then plays an important structural (hub-like) role for the connectivity of the graph, but it is still independent.%
   \COMMENT{}

External connectivity of a set of vertices in the ambient graph had been considered before in the context of \td s and tangles~\cite{excludedGrid, ReedConnectivityMeasure}.%
   \COMMENT{}
   But it was Dunwoody and Kr\"on who realized that $k$-inseparability can serve to extend Tutte's result to $k>2$: they showed that the $k$-blocks of a finite $k$-connected graph can, in principle, be separated canonically in a tree-like way~\cite{DunwoodyKroenArXiv}. We shall re-prove this in a simpler and stronger form,%
   \COMMENT{}
   extend it to graphs of arbitrary connectivity, and cast the `tree-like way' in the standard form of \td s. We show in particular that every finite graph has a canonical \td\ of adhesion at most~$k$ such that distinct $k$-blocks are contained in different parts (Theorem~1); this appears to solve the original problem for fixed $k$ in a strongest-possible way. For graphs whose $k$-blocks have size at least~$3k/2$ for all~$k$, a~weak but necessary additional assumption, these decompositions can be combined into one unified \td\ that distinguishes all the blocks of the graph, simultaneously for all~$k$ (Theorem~2).

Our paper is independent of the results stated in~\cite{DunwoodyKroenArXiv}.%
   \footnote{The starting point for this paper was that, despite some effort, we were unable to verify some of the results claimed in~\cite{DunwoodyKroenArXiv}.}%
   \COMMENT{}
   Our approach will be as follows. We first develop a more general theory of separation systems to deal with the following abstract problem. Let \cS\ be a set of separations in a graph, and let \cI\ be a collection of \emph{\cS-inseparable} sets of vertices, sets which, for every separation $\AB\in\cS$, lie entirely in $A$ or entirely in~$B$. Under what condition does \cS\ have a \emph{nested} subsystem \cN\ that still separates all the sets in~\cI? In a further step we show how such nested separation systems \cN\ can be captured by \td s.%
   \footnote{It is easy to see that \td s give rise to nested separation systems. The converse is less clear.}

The gain from having an abstract theory of how to extract nested subsystems from a given separation system is its flexibility. For example, we shall use it in~\cite{profiles} to prove that every finite graph has a canonical (in the sense above) \td\  separating all its maximal tangles. This improves on the result of Robertson and Seymour~\cite{GMX} mentioned earlier, in that their decomposition is not canonical in our sense: it depends on an assumed vertex enumeration to break ties when choosing which of two crossing separations should be picked for the nested subsystem. The choices made by our decompositions will depend only on the structure of the graph. In particular, they will be invariant under its automorphisms, which thus act naturally also on the set of parts of the decomposition and on the associated decomposition tree.

To state our main results precisely, let us define their terms more formally. In addition to the terminology explained in~\cite{DiestelBook10noEE} we say that a set $X$ of vertices in a graph~$G$ is \emph{$k$-inseparable} in~$G$ if $|X|>k$%
   \COMMENT{}
   and no set $S$ of at most $k$ vertices separates two vertices of $X\sm S$ in~$G$. A~maximal $k$-inseparable set of vertices is a \emph{$k$-block},%
   \footnote{For reasons of compatibility with tangles and other concepts in graph minor theory, we shall call this a `$(k+1)$-block' in future papers, rather than a `$k$-block'. Thus, in future, a {\it $k$-block\/} will be a maximal $(k-1)$-inseparable set of vertices: a maximal set of at least $k$ vertices such that no two of them are separated by $<k$ vertices in~$G$. The assertions of Theorems 1 and~2 below will change only in that the adhesion of their \td s will change to $<k$, and we encourage readers to use the new definition and adapted statements of Theorems 1 and~2 should they wish to cite them. However, for reading the current paper it will be perfect to keep with the old definition of a $k$-block to avoid confusion. See~\cite{ForcingBlocks} for examples of $k$-blocks.\looseness=-1}%
   \COMMENT{}
   or simply a {\em block\/}. The smallest $k$ for which a block is a $k$-block is the {\em rank} of that block; the largest such $k$ is its {\em order\/}.

The intersections $V_t\cap V_{t'}$ of `adjacent' parts in a \td\ $(\cT, \cV)$ of $G$ (those for which $tt'$ is an edge of~$\cT$) are the {\em adhesion sets\/} of~$(\cT, \cV)$; the maximum size of such a set is the {\em adhesion\/} of $(\cT, \cV)$. A~\td\ of adhesion at most~$k$%
   \COMMENT{}
   {\em distinguishes\/} two $k$-blocks $b_1,b_2$ of $G$ if they are contained in different parts, $V_{t_1}$ and~$V_{t_2}$ say. It does so {\em efficiently\/} if the $t_1$--$t_2$ path in the decomposition tree~$\cT$ has an edge $tt'$ whose adhesion set (which will separate $b_1$ from~$b_2$ in~$G$) has size $\kappa(b_1,b_2)$, the minimum size of a $b_1$--$b_2$ separator in~$G$. The \td\ $(\cT, \cV)$ is $\Aut(G)$-{\em invariant\/} if the automorphisms of $G$ act on the set of parts in a way that induces an action on the tree~\cT.%
   \COMMENT{}

\begin{thm1}
Given any integer $k\ge 0$, every finite graph $G$ has an $\Aut(G)$-invariant \td\ of adhesion at most~$k$ that efficiently distinguishes all its $k$-blocks.
\end{thm1}

Unlike in the original problem, the graph $G$ in Theorem~1 is not required to be $k$-connected. This is a more substantial improvement than it might seem. It becomes possible only by an inductive approach which refines, for increasing $\ell=0,1,\dots$, each part of a given \td\ of $G$ of adhesion at most~$\ell$ by a finer \td\ of adhesion at most~$\ell+1$, until for $\ell=k$ the desired decomposition is achieved. The problem with this approach is that, in general, a graph $G$~need not admit a unified \td\ that distinguishes its $\ell$-blocks for all $\ell\in\N$ simultaneously. Indeed, we shall see in Section~\ref{sec_kinsep} an  example where $G$ has two $\ell$-blocks separated by a unique separation of order at most~$\ell$, as well as two $(\ell+1)$-blocks separated by a unique separation of order at most~$\ell+1$, but where these two separations `cross': we cannot adopt both for the same \td\ of~$G$. The reason why this inductive approach nonetheless works for a proof of Theorem~1 is that we aim for slightly less there: at stage $\ell$ we only separate those $\ell$-blocks of $G$ that contain a $k$-block for the fixed $k$ given in the theorem, not all the $\ell$-blocks of~$G$.

However, there is a slight strengthening of the notion of a block that does make it possible to construct an overall \td\ separating all the blocks of a graph at once. We shall call such blocks {\em robust\/}. Their precise definition is technical and will be given later; it essentially describes the exact way in which the offending block of the above counterexample lies in the graph.%
   \footnote{Thus we shall prove that our counterexample is essentially the only one: all graphs not containing it have a unified \td\ distinguishing all their blocks.}
  In practice `most' blocks of a graph will be robust, including all $k$-blocks that are complete or have size at least $3k/2$. 

If all the blocks of a graph~$G$ are robust, how will they lie in the unified \td\ of $G$ that distinguishes them all? Some blocks (especially those of large order) will reside in a single part of this decomposition, while others (of smaller order) will inhabit a subtree consisting of several parts. Subtrees accommodating distinct $k$-blocks, however, will be disjoint. Hence for any fixed~$k$ we can contract them to single nodes, to reobtain the \td\ from Theorem~1 in which the $k$-blocks (for this fixed~$k$) inhabit distinct single parts. As $k$ grows, we thus have a sequence $\left(\cT_k,\cV_k\right)_{k\in\N}$ of \td s, each refining the previous, that gives rise to our overall \td\ in the last step of the sequence.

Formally, let us write $\left(\cT_m,\cV_m\right) \minor \left(\cT_n,\cV_n\right)$ for \td s $\left(\cT_m,\cV_m\right)$ and $\left(\cT_n,\cV_n\right)$ if the decomposition tree $\cT_m$ of the first is a minor of the decomposition tree~$\cT_n$ of the second, and a part $V_t\in\cV_m$ of the first  decomposition is the union of those parts $V_{t'}$ of the second whose nodes $t'$ were contracted to the node $t$ of~$\cT_m$.

\begin{thm2}
For every finite graph $G$ there is a sequence $\left(\cT_k,\cV_k\right)_{k\in\N}$ of tree-decompositions such that, for all~$k$,
\begin{enumerate}[\rm (i)]
\item $\left(\cT_k,\cV_k\right)$ has adhesion at most $k$ and distinguishes all robust $k$-blocks;%
   \COMMENT{}
\item $\left(\cT_k,\cV_k\right) \minor \left(\cT_{k+1},\cV_{k+1}\right)$;
\item $\left(\cT_k,\cV_k\right)$ is $\Aut(G)$-invariant.
\end{enumerate}
\end{thm2}

\noindent
The decomposition $\left(\cT_k,\cV_k\right)$ will in fact distinguish distinct robust $k$-blocks $b_1,b_2$ efficiently, by (i) for $k'= \kappa(b_1,b_2)$%
   \COMMENT{}
   and (ii). In Section~\ref{sec_kinsep} we shall prove Theorem~2 in a stronger form, which also describes how blocks of different rank or order are distinguished.

This paper is organized as follows. In Section~\ref{sec_nested} we collect together some properties of pairs of separations, either crossing or nested. In Section~\ref{sec_tree} we define a structure tree \cT\ associated canonically with a nested set of separations of a graph~$G$. In Section~\ref{sec_td} we construct a \td\ of~$G$ modelled on~\cT, and study its parts. In Section~\ref{sec_sys} we find conditions under which, given a set \cS\ of separations and a collection \cI\ of \cS-inseparable set of vertices, there is a nested subsystem of \cS\ that still separates all the sets in~\cI. In Section~\ref{sec_kinsep}, finally, we apply all this to the case of $k$-separations and $k$-blocks. We shall derive a central result, Theorem~\ref{thm_kinsep_sys}, which includes Theorems 1 and~2 as special cases.

\section{Separations} \label{sec_nested}
Let $G = (V,E)$ be a finite graph. A \emph{separation} of~$G$ is an ordered pair \AB~such that $A,B \sub V$ and $G[A] \cup G[B] = G$. A separation \AB\ is \emph{proper} if neither $A\sm B$ nor $B\sm A$ is empty. The \emph{order} of a separation \AB\ is the cardinality of its \emph{separator} \sep AB; the sets $A,B$ are its {\em sides\/}. A~separation of order $k$ is a \emph{$k$-separation}.

A separation \AB~\emph{separates} a set $I \sub V$ if $I$ meets both $A \sm B$ and $B \sm A$.
   \COMMENT{}
Two sets $I_0, I_1$ are \emph{weakly separated} by a separation \AB~if $I_i \sub A$ and $I_{1-i} \sub B$ for an $i \in 
\{0,1\}$. They are \emph{properly separated}, or simply \emph{separated}, by \AB\ if in addition neither $I_0$ nor $I_1$ is contained in \sep AB.

Given a set $\cS$ of separations, we call a set of vertices \emph{\cS-inseparable} if no separation in \cS~separates it. A maximal \cS-inseparable set of vertices is an \emph{\cS-block}, or simply a \emph{block} if \cS\ is fixed in the context.

\begin{lem}\label{block_lem}
Distinct \cS-blocks $b_1,b_2$ are  separated by some $\AB\in\cS$.
\end{lem}

\begin{proof}
Since $b_1$ and $b_2$ are maximal \cS-inseparable sets, $b:=b_1 \cup b_2$ can be separated by some $\AB\in\cS$. Then $b\sm B\neq\emptyset\neq b\sm A$, but being \cS-inseparable, $b_1$~and $b_2$ are each contained in $A$ or~$B$. Hence $(A,B)$ separates $b_1$ from~$b_2$.
\end{proof}

A set of vertices is \emph{small} with respect to \cS\ if it is contained in the separator of some separation in \cS. If \cS\ is given from the context, we simply call such a set \emph{small}. Note that if two sets are weakly but not properly separated by some separation in \cS\ then at least one of them is small.

Let us look at how different separations of~$G$ can relate to each other. The set of all separations of $G$ is partially ordered by
\begin{equation}\label{def_PO}
\AB \le \CD\ :\Lra\ A\sub C\ \text{and}\ B \supseteq D.
\end{equation}
Indeed, reflexivity, antisymmetry and transitivity follow easily from the corresponding properties of set inclusion on~$\cP(V)$. Note that changing the order in each pair reverses the relation:
\begin{equation}\label{eq_order_flip}
\AB \le \CD\ \Lra\ \BA \ge  \DC.
\end{equation}
Let \CD\ be any proper separation.
\begin{txteq}\label{comp}
No proper separation \AB\ is $\le$-comparable with both \CD\ and~\DC. In particular, $\CD\not\le\DC$.
\end{txteq}
Indeed, if $\AB \le \CD$ and also $\AB\le\DC$, then $A\sub C\sub B$ and hence $A \sm B = \emptyset$, a contradiction. By~\eqref{eq_order_flip}, the other cases all reduce to this case by changing notation: just swap $(A,B)$ with $\BA$ or \CD\ or~\DC.%
   \COMMENT{}

\medbreak

The way in which two separations relate to each other can be illustrated by a \emph{cross-diagram} as in Figure \ref{fig:corners}. In view of such diagrams, we introduce the following terms for any set $\{\AB,\CD\}$ of two separations, not necessarily distinct. The set $A \cap B \cap C \cap D$ is their \emph{centre}, and \sep AC, \sep AD, \sep BC, \sep BD are their \emph{corners}. The corners \sep AC and \sep BD are \emph{opposite\/}, as are the corners \sep AD and \sep BC. Two corners that are not opposite are \emph{adjacent}. The {\em link} between two adjacent corners is their intersection minus the centre. A~corner minus its links and the centre is the \emph{interior} of that corner; the rest~-- its two links and the centre~-- are its \emph{boundary}. We shall write $\partial K$ for the boundary of a corner~$K$.

   \begin{figure}[htpb]
\centering
   	  \includegraphics{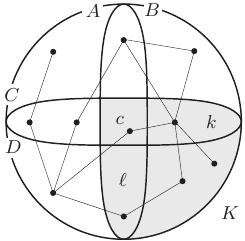}
   	  \caption{The cross-diagram $\{\AB,\CD\}$ with centre~$c$ and a~corner~$K$ and its links~$k,\ell$.}
   \label{fig:corners}\vskip-12pt\vskip0pt
   \end{figure}

A corner forms a separation of $G$ together with the union of the other three corners. We call these separations \emph{corner separations}. For example, \csepn ACBD\ (in this order)%
   \COMMENT{}
 is the corner separation for the corner \sep AC in~$\{\AB,\CD\}$.

The four corner separations of a cross-diagram compare with the two separations forming it, and with the inverses of each other, in the obvious way:
\begin{txteq}\label{cornerlessthanhalf}
Any two separations \AB, \CD\ satisfy $(A\cap C, B\cup D) \le \AB$.
\end{txteq}
 \vskip-1.5\baselineskip
\begin{txteq}\label{cornersnested}
If $\sepn IJ$ and $\sepn KL$ are distinct corner separations of the same cross-diagram, then $\sepn IJ\le\sepn LK$.
\end{txteq}%
   \COMMENT{}

\sloppy
Inspection of the cross-diagram for \AB\ and~\CD\ shows that ${\AB \le \CD}$ if and only if the corner \sep AD has an empty interior and empty links, i.e., the entire corner \sep AD is contained in the centre:
\begin{equation}\label{nested}
\AB \le \CD\ \Lra\ \sep AD \sub \sep BC.
\end{equation}
\fussy
Another consequence of $\AB\le\CD$ is that $\sep AB \sub C$ and ${\sep CD\sub B}$. So both separators live entirely on one side of the other separation.%
   \COMMENT{}

A separation \AB\ is \emph{tight} if every vertex of \sep AB has a neighbour in $A\sm B$ and another neighbour in~$B\sm A$. For tight separations, one can establish that $\AB\le\CD$ by checking only one of the two inclusions in~\eqref{def_PO}:
\begin{txteq}\label{tightorder}
If \AB\ and \CD\ are separations such that $A\sub C$ and \CD\ is tight, then $\AB\le\CD$.
\end{txteq}
Indeed, suppose $D\not\subseteq B$. Then as $A\sub C$, there is a vertex $x\in (C\cap D)\sm B$. As  \CD\ is tight, $x$ has a neighbour $y\in D\sm C$, but since $x\in A\sm B$ we see that $y\in A$. So $A\sm C\neq\emptyset$, contradicting our assumption.

\medbreak

Let us call \AB\ and \CD\ \emph{nested}, and write $\AB\|\CD$, if \AB\ is comparable with \CD\ or with~\DC\ under~$\le$. By~\eqref{eq_order_flip}, this is a symmetrical relation. For example, we saw in \eqref{cornerlessthanhalf} and~\eqref{cornersnested} that the corner separations of a cross-diagram are nested with the two separations forming it, as well as with each other.

Separations \AB\ and \CD\ that are not nested are said to~\emph{cross}; we then write $\AB\nparallel\CD$.

Nestedness is invariant under `flipping' a separation: if $\AB\|\CD$ then also $\AB\|\DC$, by definition of~$\|$, but also $\BA\|\CD$ by~\eqref{eq_order_flip}. Thus although nestedness is defined on the separations of~$G$, we may think of it as a symmetrical relation on the unordered pairs $\{A,B\}$ such that \AB\ is a separation.%
   \COMMENT{}

By~\eqref{nested}, nested separations have a simple description in terms of cross-diagrams:
\begin{txteq}\label{emptycorner}
Two separations are nested if and only if one of their four corners has an empty interior and empty links.
\end{txteq}
In particular:
\begin{txteq}\label{cross-separators}
Neither of two nested separations separates the separator of the other\rlap{.}
\end{txteq}
The converse of \eqref{cross-separators} fails only if there is a corner with a non-empty interior whose links are both empty.

Although nestedness is reflexive and symmetric, it is not in general transitive. However when transitivity fails, we can still say something:%
   \COMMENT{}

\begin{lem}\label{lem_nested_or}
If $\AB\|\CD$ and $\CD\|\EF$ but $\AB\nparallel\EF$, then \CD\ is nested with every corner separation of $\{\AB,\EF\}$, and for one corner separation $\sepn IJ$ we have either $\CD \le \sepn IJ$ or ${\DC \le \sepn IJ}$.
\end{lem}

\begin{proof}
Changing notation as necessary,%
  \COMMENT{}
   we may assume that $\AB\le\CD$, and that \CD\ is comparable with~\EF.%
   \footnote{Note that such change of notation will not affect the set of corner separations of the cross-diagram of \AB\ and~\EF, nor the nestedness (or not) of \CD\ with those corner separations.}
   If $\CD\le\EF$ we have $\AB\le\EF$, contrary to our assumption. Hence $\CD\ge\EF$, or equivalently by~\eqref{eq_order_flip}, $\DC\le\FE$. As also $\DC\le\BA$, we thus have $D\sub F\cap B$ and $C\supseteq E\cup A$ and therfore $$\DC\le\csepn FBEA\specrel\le{\eqref{cornersnested}} \sepn LK$$
for each of the other three corner separations $\sepn KL$ of $\{\AB,\EF\}$.
\end{proof}

   \begin{figure}[htpb]
\centering
   	  \includegraphics{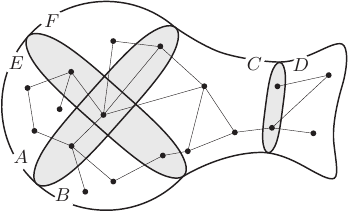}
   	  \caption{Separations as in Lemma~\ref{lem_nested_or}}
   \label{fig:nontrans}\vskip-12pt\vskip0pt
   \end{figure}

Figure~\ref{fig:nontrans} shows an example of three separations witnessing the non-transiti\-vi\-ty of nestedness. Its main purpose, however, is to illustrate the use of Lemma~\ref{lem_nested_or}. We shall often be considering which of two crossing separations, such as \AB\ and \EF\ in the example, we should adopt for a desired collection of nested separations already containing some separations such as~\CD. The lemma then tells us that we can opt to take neither, but instead choose a suitable corner separation.

Note that there are two ways in which three separations can be pairwise nested. One is that they or their inverses form a chain under~$\le$. But there is also another way, which will be important later; this is illustrated in Figure~\ref{fig:nested}.

   \begin{figure}[htpb]
\centering\vskip-6pt\vskip0pt
   	  \includegraphics{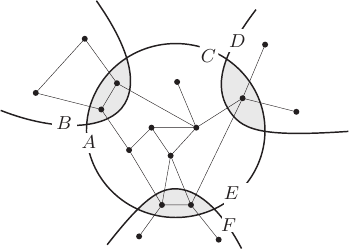}
   	  \caption{Three nested separations not coming from a $\le$-chain}
   \label{fig:nested}\vskip-6pt\vskip0pt
   \end{figure}

We need one more lemma.

\penalty-999

\begin{lem}\label{lem_block_sep}
Let \cN\ be a set of separations of $G$ that are pairwise nested. Let \AB\ and \CD\ be two further separations, each nested with all the separations in~\cN. Assume that $\AB$ separates an \cN-block~$b$, and that $\CD$ separates an \cN-block $b' \neq b$. Then $\AB\|\CD$. Moreover, $\sep AB\sub b$ and $C\cap D\sub b'$.
\end{lem}

\begin{proof}
By Lemma~\ref{block_lem}, there is a separation $\EF\in \cN$ with $b \sub E$ and ${b' \sub F}$. Suppose $\AB \nparallel \CD$. By symmetry and Lemma~\ref{lem_nested_or} we may assume that
$$\EF \le \csepn ACBD.$$
But then $b \sub E \sub \sep AC \sub A$, contradicting the fact that \AB\ separates~$b$. Hence $\AB\|\CD$, as claimed.

If $A\cap B\not\sub b$, then there is a $(K,L)\in\cN$ which separates $b\cup (A\cap B)$. We may assume that $b\sub L$ and that $A\cap B\not\sub L$. The latter implies that $(K,L)\not \leq \AB$%
   \COMMENT{}
   and  $(K,L)\not\leq \BA$. So $(K,L)\|\AB$%
   \COMMENT{}
   implies that either $(L,K)\leq \AB$ or $(L,K)\leq \BA$. Thus $b\sub L\sub A$ or $b\sub L\sub B$, a contradiction to the fact that \AB\ separates $b$. Similarly we obtain $\sep CD \sub b'$.
\end{proof}

\section{Nested separation systems and tree structure} \label{sec_tree}

A set \cS\ of separations is \emph{symmetric} if $\AB\in\cS$ implies ${\BA\in\cS}$, and \emph{nested} if every two separations in \cS\ are nested. Any symmetric set of proper separations is a \emph{\sys}. Throughout this section and the next, we consider a fixed nested \sys~\cN\ of our graph~$G$.

Our aim in this section will be to describe \cN\ by way of a \emph{structure tree} $\cT=\T$, whose edges will correspond to the separations in~\cN. Its nodes%
   \footnote{While our graphs $G$ have {\em vertices\/}, structure trees will have {\em nodes\/}.}
    will correspond to subgraphs of~$G$. Every automorphism of~$G$ that leaves \cN\ invariant will also act on~\cT. Although our notion of a separation system differs from that of Dunwoody and Kr\"on~\cite{DunwoodyKroenArXiv, CuttingUpGraphs}, the main ideas of how to describe a nested system by a structure tree can already be found there.

Our main task in the construction of~\cT\ will be to define its nodes. They will be the equivalence classes of the following equivalence relation $\sim$ on~\cN, induced by the ordering $\le$ from~\eqref{def_PO}: 
\begin{equation}\label{def_sim}
\AB\sim\CD\ :\Lra \alt{\!\!\!\AB = \CD \mbox{ or}}{\!\!\!\BA\pre\CD\mbox{ in }(\cN,\le).}
\end{equation}
(Recall that, in a partial order $(P,\le)$, an element $x\in P$ is a {\em predecessor\/} of an element $z\in P$ if $x<z$ but there is no $y\in P$ with $x < y < z$.)

Before we prove that this is indeed an equivalence relation, it may help to look at an example: the set of vertices in the centre of Figure~\ref{fig:nested} will be the node of \cT\ represented by each of the equivalent nested separations \AB, \CD\ and~\EF.

\goodbreak

\begin{lem}\label{lem_eq}
The relation $\sim$ is an equivalence relation on \cN.
\end{lem}

\begin{proof}
Reflexivity holds by definition, and symmetry follows from (\ref{eq_order_flip}). To show transitivity assume that $\AB\sim\CD$ and $\CD\sim\EF$, and that all these separations are distinct. Thus,
\begin{enumerate}[(i)]
\item \BA\pre\CD;
\item \DC\pre\EF.
\end{enumerate}
And by~(\ref{eq_order_flip}) also
\begin{enumerate}[(i)]
\setcounter{enumi}{2}
\item \DC\pre\AB;
\item \FE\pre\CD.
\end{enumerate}

\goodbreak

By (ii) and (iii), \AB\ is incomparable with \EF. Hence, since \cN\ is nested, \BA\ is comparable with \EF. If $\EF\leq\BA$ then by (i) and~(ii), 
$\DC\leq\CD$, which contradicts~\eqref{comp} (recall that all separations in a separation system are required to be proper). Thus $\BA<\EF$, as desired.

Suppose there is a separation $\XY\in\cN$ with $\BA<\XY<\EF$. As \cN\ is nested, \XY~is comparable with either \CD\ or \DC. By (i) and~(ii), $\XY\not <\CD$ and $ \DC\not <\XY$. Now if $\CD\le\XY<\EF$ then by (iv), \CD\ is comparable to both \EF\ and \FE, contradicting~\eqref{comp}.
Finally, if $\DC\geq\XY> \BA$, then by (iii), \DC\ is comparable to both \BA\ and \AB, again contradicting~\eqref{comp}.%
  \COMMENT{}
   We have thus shown that \BA\pre\EF, implying that $\AB\sim\EF$ as claimed.
\end{proof}

Note that, by~\eqref{comp}, the definition of equivalence implies:%
   \COMMENT{}
\begin{txteq}\label{incomp}
Distinct equivalent proper separations are incomparable under~$\le$.
\end{txteq}

We can now define the nodes of $\cT=\T$ as planned, as the equivalence classes of~$\sim\,$:
 $$V(\cT) := \big\{ [\AB] : \AB\in\cN\big\}.$$\medbreak

Having defined the nodes of~$\cT$, let us define its edges. For every separation $\AB\in\cN$ we shall have one edge, joining the nodes represented by \AB\ and~\BA, respectively. To facilitate notation later, we formally give \cT\ the abstract edge set
 $$E(\cT) := \big\{\{\AB,\BA\}\mid\AB\in\cN\big\}$$
and declare an edge $e$ to be incident with a node $\cX\in{V(\cT)}$ whenever $e\cap\cX \neq \emptyset$ (so that the edge $\{\AB,\BA\}$ of~\cT\ joins its nodes $[\AB]$ and $[\BA]$). We have thus, so far, defined a multigraph~$\cT$.

As $\AB\not\sim\BA$ by definition of $\sim$, our multigraph
 \cT~has no loops. Whenever an edge $e$ is incident with a node~$\cX$, the non-empty set $e\cap\cX$ that witnesses this is a singleton set containing one separation.%
   \COMMENT{}
   We denote this separation by~$(e\cap\cX)$. Every separation $\AB\in\cN$ occurs as such an $(e\cap\cX)$, with $\cX = [\AB]$ and $e = \{\AB,\BA\}$. Thus,
\begin{txteq}\label{nodes}
Every node \cX\ of~\cT\ is the set of all the separations $(e\cap\cX)$ such that $e$~is incident with~\cX. In particular, $\cX$~has degree~$|\cX|$ in~\cT.
\end{txteq}

Our next aim is to show that $\cT$ is a tree.

\begin{lem}\label{lem_walk}
Let $W = \cX_1e_1\cX_2e_2\cX_3$ be a walk in \cT\ with $e_1 \neq e_2$. Then $(e_1\cap\cX_1)$ is a predecessor of~$(e_2\cap\cX_2)$.
\end{lem}

\begin{proof}
Let $(e_1 \cap \cX_1) = \AB$ and $(e_2 \cap \cX_2) = \CD$. Then $\BA = (e_1 \cap \cX_2)$ and $\BA \sim \CD$. Since $e_1 \neq e_2$ we have $\BA\neq\CD$. Thus, $\AB$~is a predecessor of~$\CD$ by definition of~$\sim$.
\end{proof}

And conversely:

\sloppy
\begin{lem}\label{path_in_T}
Let $(E_0,F_0), \dots , (E_k,F_k)$ be separations in~\cN\ such that each $(E_{i-1},F_{i-1})$ is a predecessor of $(E_i,F_i)$ in~$(\cN,\le)$. Then $[(E_0,F_0)], \dots, [(E_k,F_k)]$ are the nodes of a walk in~\cT, in this order.%
   \COMMENT{}
\end{lem}

\begin{proof}
By definition of~$\sim$, we know that $(F_{i-1},E_{i-1})\sim (E_i,F_i)$. Hence for all $i=1,\dots,k$, the edge $\{(E_{i-1},F_{i-1}), (F_{i-1}, E_{i-1})\}$ of~\cT\ joins the node $[(E_{i-1}, F_{i-1})]$ to the node $[(E_i, F_i)] = [(F_{i-1}, E_{i-1})]$.
\end{proof}

\fussy
\begin{thm}\label{thm_tree}
The multigraph \T\ is a tree.
\end{thm}
\begin{proof}
We have seen that $\cT$ is loopless. Suppose that \cT\ contains a cycle $\cX_1 e_1 \cdots \cX_{k-1} e_{k-1} \cX_k$, with ${\cX_1 = \cX_k}$ and $k>2$.
Applying Lemma \ref{lem_walk} $(k-1)$ times yields
 $$\AB := (e_1 \cap \cX_1) < \ldots < (e_{k-1} \cap \cX_{k-1}) <  (e_{1} \cap \cX_{k})  = \AB,$$
a contradiction. Thus, \cT~is acyclic; in particular, it has no parallel edges.


It remains to show that \cT\ contains a path between any two given nodes $[\AB]$ and~$[\CD]$. As $\cN$ is nested, we know that \AB\ is comparable with either \CD\ or \DC. Since $[\CD]$ and $[\DC]$ are adjacent, it suffices to construct a walk between [\AB] and one of them. Swapping the names for $C$ and $D$ if necessary, we may thus assume that \AB\ is comparable with~\CD. Reversing the direction of our walk if necessary, we may further assume that $\AB<\CD$. Since our graph $G$ is finite, there is a chain
 $$\AB = (E_0,F_0) < \cdots <(E_k,F_k) = \CD$$
such that $(E_{i-1},F_{i-1})$ is a predecessor of $(E_i,F_i)$, for every~$i = 1,\dots,k$. By Lemma~\ref{path_in_T}, \cT~contains the desired path from $[\AB]$ to~$[\CD]$.
\end{proof}

\begin{cor}\label{cor_tn_Gamma}
If \cN\ is invariant under a group $\Gamma \le \Aut(G)$ of automorphisms of~$G$, then $\Gamma$ also acts on $\cT$ as a group of automorphisms.
\end{cor}

\begin{proof}
Any automorphism $\alpha$ of $G$ maps separations to separations, and preserves their partial ordering defined in~\eqref{def_PO}. If both $\alpha$ and $\alpha^{-1}$ map separations from~\cN\ to separations in~\cN, then $\alpha$ also preserves the equivalence of separations under~$\sim$.%
   \COMMENT{}
   Hence~$\Gamma$, as stated, acts on the nodes of~\cT\ and preserves their adjacencies and non-adjacencies.
\end{proof}

\section{From structure trees to \td s}\label{sec_td}

Throughout this section, \cN\ continues to be an arbitrary nested separation system of our graph~$G$. Our aim now is to show that $G$ has a \td, in the sense of Robertson and Seymour, with the structure tree~$\cT=\T$ defined in Section~\ref{sec_tree} as its decomposition tree. The separations of $G$ associated with the edges of this decomposition tree%
   \footnote{as in the theory of \td s, see e.g.~\cite[Lemma 12.3.1]{DiestelBook10noEE}}
   will be precisely the separations in~\cN\ identified by those edges in the original definition of~\cT.

Recall that a \emph{\td} of $G$ is a pair $(T,\cV)$ of a tree $T$ and a family $\cV=(V_t)_{t\in T}$ of vertex sets $V_t\sub V(G)$, one for every node of~$T$, such that:
\begin{enumerate}[(T1)]
\item $V(G) = \bigcup_{t\in T}V_t$;
\item for every edge $e\in G$ there exists a $t \in T$ such that both ends of $e$ lie in~$V_t$;
\item $V_{t_1} \cap V_{t_3} \sub V_{t_2}$ whenever $t_2$ lies on the $t_1$--$t_3$ path in~$T$.
\end{enumerate}

To define our desired \td~$(\cT,\cV)$, we thus have to define the family $\cV = (V_\cX)_{\cX\in{V(\cT)}}$ of its parts: with every node $\cX$ of~\cT\ we have to associate a set $V_\cX$ of vertices of~$G$. We define these as follows:
\begin{equation}\label{VtDef}
 V_\cX := \bigcap\big\{\,A\ |\ \AB \in\cX\,\big\}
\end{equation}

\begin{ex}\label{block-cutvx}
Assume that $G$ is connected, and consider as \cN\ the nested set of all proper 1-separations \AB\ and~\BA\ such that $A\sm B$ is connected in~$G$.%
   \COMMENT{}
   Then \cT\ is very similar to the block-cutvertex tree of~$G$: its nodes will be the blocks in the usual sense (maximal 2-connected subgraphs or bridges) plus those cutvertices that lie in at least three blocks.%
   \COMMENT{}

   \begin{figure}[htpb]
\centering
   	  \includegraphics{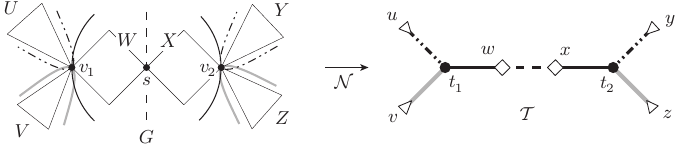}
   	  \caption{\cT~has an edge for every separation in~\cN. Its nodes correspond to the blocks and some of the cutvertices of~$G$.}
   \label{fig:bc1}\vskip-6pt\vskip0pt
   \end{figure}

In Figure~\ref{fig:bc1}, this separation system \cN~contains all the proper 1-separations of~$G$. The separation \AB\ defined by the cutvertex~$s$, with $A:= U\cup V\cup W$ and $B:= X\cup Y\cup Z$ say, defines the edge $\{\AB, \BA\}$ of~\cT\ joining its nodes $w = [\AB]$ and $x = [\BA]$.%
   \COMMENT{}

\begin{figure}[htpb]
\centering
   	  \includegraphics{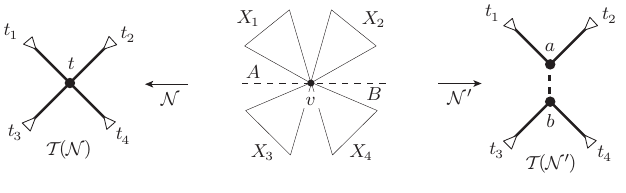}
   	  \caption{$\cT' = \cT(\cN')$ has distinct nodes $a,b$ whose parts in the \td\ $(\cT',\cV)$ coincide: $V_a = \{v\} = V_b$.}
   \label{fig:bc2}\vskip-6pt\vskip0pt
   \end{figure}

In Figure~\ref{fig:bc2} we can add to \cN\ one of the two crossing 1-separations not in~\cN\ (together with its inverse), to obtain a set $\cN'$ of separations that is still nested. For example, let
 $$\cN':= \cN\cup\{\AB,\BA\}$$
with $A:= X_1\cup X_2$ and $B:= X_3\cup X_4$. This causes the central node~$\cX$ of \cT\ to split into two nodes $a = [\AB]$ and $b=[\BA]$ joined by the new edge $\{\AB,\BA\}$. However the new nodes $a,b$ still define the same part of the \td\ of~$G$ as $t$ did before: $V_a = V_b = V_t = \{v\}$.
\end{ex}

Before we prove that $(\cT,\cV)$ is indeed a \td, let us collect some information about its parts $V_\cX$, the vertex sets defined in~\eqref{VtDef}.

\begin{lem}\label{Vt_insep}
Every 
$V_\cX$ is \cN-inseparable.
\end{lem}

\begin{proof}
Let us show that a given separation $\CD\in\cN$ does not separate~$V_t$. Pick $\AB\in\cX$. Since \cN\ is nested, and swapping the names of $C$ and $D$ if necessary, we may assume that \AB\ is $\le$-comparable with \CD. If ${\AB\le\CD}$ then $V_t \sub A\sub C$, so \CD\ does not separate~$V_t$. If $\CD<\AB$, there is a $\le$-predecessor \EF\ of \AB\ with $\CD\le\EF$. Then $\FE\sim\AB$ and hence $V_t\sub F \sub D$, so again \CD\ does not separate~$V_t$.
\end{proof}

The sets $V_\cX$ will come in two types: they can be

\begin{itemize}
\item \cN-blocks (that is, maximal \cN-inseparable sets of vertices), or

\item `hubs' (defined below).
\end{itemize}
Nodes $\cX\in\cT$ such that $V_t$ is an \cN-block are \emph{block nodes\/}. A~node $\cX\in\cT$ such that $V_t = A\cap B$ for some $\AB\in t$ is a \emph{hub node} (and $V_t$ a \emph{hub}).%
   \COMMENT{}

In Example~\ref{block-cutvx}, the \cN-blocks were the (usual) blocks of~$G$; the hubs were singleton sets consisting of a cutvertex. Example~\ref{ex:td} will show that \cX\ can be a hub node and a block node at the same time. Every hub is a subset of a block: by~\eqref{cross-separators}, hubs are \cN-inseparable, so they extend to maximal \cN-inseparable sets.

Hubs can contain each other properly (Example~\ref{ex:td} below). But a hub $V_t$ cannot be properly contained in a separator $A\cap B$ of any $\AB\in t$. Let us prove this without assuming that $V_t$ is a hub:

\begin{lem}\label{nosmallhubs}
Whenever $\AB\in t\in\cT$, we have $A\cap B\sub V_t$. In particular, if $V_t\sub A\cap B$,  then $V_t = A\cap B$ is a hub with hub node~$t$.
\end{lem}

\begin{proof}
Consider any vertex $v\in (A\cap B)\sm V_t$. By definition of~$V_t$, there exists a separation $\CD\in t$ such that $v\notin C$. This contradicts the fact that $B\sub C$ since $\AB\sim\CD$.
\end{proof}

\begin{lem}\label{lem_eq_block}
Every node of \cT\ is either a block node or a hub node.
\end{lem}

\begin{proof}
Suppose $\cX\in\cT$ is not a hub node; we show that $\cX$ is a block node. By Lemma~\ref{Vt_insep}, $V_\cX$~is \cN-inseparable. We show that $V_t$~is maximal in $V(G)$ with this property: that for every vertex $x\notin V_t$ the set $V_t\cup\{x\}$ is not \cN-inseparable.

By definition of~$V_t$, any vertex $x\notin V_t$ lies in $B\sm A$ for some $(A,B)\in t$. Since \cX\ is not a hub node,  Lemma~\ref{nosmallhubs} implies that $V_t\not\sub A\cap B$. As $V_t\sub A$, this means that $V_t$ has a vertex in~$A\sm B$. Hence $(A,B)$ separates  $V_t\cup\{x\}$, as desired.
\end{proof}

Conversely, all the \cN-blocks of $G$ will be parts of our \td:

\begin{lem}\label{blocksarenodes}
Every \cN-block is the set $V_\cX$ for a node $\cX$ of \cT.
\end{lem}
\begin{proof}
Consider an arbitrary \cN-block $b$.

Suppose first that $b$ is small. Then there exists a separation $\AB\in\cN$ with $b\sub\sep AB$. As \cN\ is nested, $A\cap B$ is \cN-inseparable by~\eqref{cross-separators}, so in fact $b = \sep AB$ by the maximality of~$b$. We show that $b = V_t$ for $t = [\AB]$. By Lemma~\ref{nosmallhubs}, it suffices to show that $V_t\sub b = A\cap B$. As $V_t\sub A$ by definition of $V_t$, we only need to show that $V_t\sub B$.
Suppose there is an $x\in V_t\sm B$. As $x\notin A\cap B = b$, the maximality of~$b$ implies that there exists a separation ${\EF\in\cN}$ such that
\begin{equation*}
F \not\supseteq b \sub E \mbox{ and } x\in F\sm E\eqno(*)
\end{equation*}
(compare the proof of Lemma~\ref{block_lem}). By~$(*)$, all corners of the cross-diagram $\{\AB,\EF\}$ other than \sep BF\ contain vertices not in the centre. Hence by~\eqref{emptycorner}, the only way in which \AB\ and \EF\ can be nested%
   \COMMENT{}
   is that $B\cap F$ does lie in the centre, i.e.\ that $\BA\le\EF$. Since $\BA\ne\EF$, by~$(*)$ and $b= A\cap B$, this means that \BA\ has a successor $\CD\le\EF$. But then $\CD\sim\AB$ and $x\notin E\supseteq C\supseteq V_t$,%
   \COMMENT{}
   a contradiction.

Suppose now that $b$ is not small. We shall prove that $b = V_t$ for $t = t(b)$, where $t(b)$ is defined as the set of separations \AB\ that are minimal with $b \sub A$. Let us show first that $t(b)$ is indeed an equivalence class, i.e., that the separations in $t(b)$ are equivalent to each other but not to any other separation in~\cN.

Given distinct $\AB,\CD\in\cX(b)$, let us show that $\AB\sim\CD$. Since both \AB\ and \CD\ are minimal as in the definition of~$t(b)$, they are incomparable. But as elements of \cN\ they are nested, so \AB\ is comparable with~\DC. If $\AB\leq\DC$ then $b\sub A\cap C \sub \sep DC$, which contradicts our assumption that $b$ is not small. Hence $\DC<\AB$. To show that \DC\ is a predecessor of~\AB, suppose there exists a separation $\EF\in\cN$ such that $\DC<\EF<\AB$. This contradicts the minimality either of \AB, if $b\sub E$, or of~\CD, if $b\sub F$. Thus, $\CD\sim\AB$ as desired.

\sloppy
Conversely, we have to show that every $\EF\in\cN$ equivalent to some ${\AB\in t(b)}$ also lies in~$t(b)$. As $\EF\sim\AB$, we may assume that ${\FE < \AB}$. Then $b\not\sub F$ by the minimality of \AB\ as an element of~$t(b)$, so $b\sub E$. To show that \EF\ is minimal with this property, suppose that $b\sub X$ also for some $\XY\in\cN$ with $\XY<\EF$.
Then \XY\ is incomparable with~\AB\,: by \eqref{incomp} we cannot have $\AB\le\XY<\EF$, and we cannot have $\XY < \AB$ by the minimality of~\AB\ as an element of~$t(b)$.
But ${\XY\|\AB}$, so \XY\ must be comparable with~\BA. Yet if $\XY\le\BA$, then $b\sub X\cap A\sub\sep BA$, contradicting our assumption that $b$ is not small, while  $\BA< \XY < \EF$ is impossible, since \BA\ is a predecessor of~\EF. 

\fussy
Hence $t(b)$ is indeed an equivalence class, i.e., $t(b)\in {V(\cT)}$. By definition of~$t(b)$, we have $b\sub \bigcap\,\{\,A\mid\AB\in t(b)\,\} =  V_{t(b)}$. The converse inclusion follows from the maximality of $b$ as an \cN-inseparable set.%
   \COMMENT{}
\end{proof}

We have seen so far that the parts $V_t$ of our intended \td\ associated with~\cN\ are all the \cN-blocks of~$G$, plus some hubs. The following proposition shows what has earned them their name:

\begin{prop}\label{funnyhub}
A hub node $\cX$ has degree at least~3 in~\cT, unless  it has the form $t = \{\AB,\CD\}$ with $A\supsetneq D$ and $B=C$ (in which case it has degree~2).
\end{prop}

\begin{proof}
Let $\AB\in t$ be such that  $V_t=A\cap B$. As $\AB\in t$ but $V_t\neq A$, we have $d(t) = |t|\ge 2$; cf.~\eqref{nodes}. Suppose that $d(t)=2$, say $t = \{\AB,\CD\}$. Then $B\sub C$ by definition of~$\sim$,%
   \COMMENT{}
   and $C\sm B = (C\cap A)\sm B = V_t\sm B = \es$ by definition of~$V_t$%
   \COMMENT{}
   and $V_t\sub A\cap B$. So $B=C$. As \AB\ and \CD\ are equivalent but not equal, this implies $D\subsetneq A$.
\end{proof}

   \begin{figure}[htpb]
\centering
   	  \includegraphics{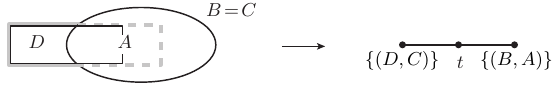}%
   \COMMENT{}
   	  \caption{A hub node $t = \{\AB,\CD\}$ of degree~2}
   \label{fig:funnyhub}\vskip-6pt\vskip0pt
   \end{figure}

Figure~\ref{fig:funnyhub} shows that the exceptional situation from Proposition~\ref{funnyhub} can indeed occur. In the example, we have $\cN = \{\AB, \BA, \CD, \DC\}$ with ${B=C}$ and $D\subsetneq A$. The structure tree $\cT$ is a path between two block nodes $\{\DC\}$ and~$\{\BA\}$ with a central hub node $\cX = \{\AB,\CD\}$,%
   \COMMENT{}
   whose set $V_\cX = A\cap B$ is not a block since it is properly contained in the \cN-inseparable set $B=C$.

Our last example answers some further questions about the possible relationships between blocks and hubs that will naturally come to mind:

   \begin{figure}[htpb]
\centering
   	  \includegraphics{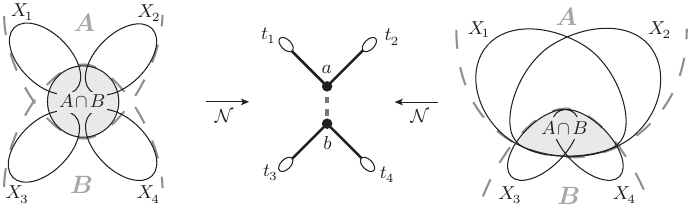}
   	  \caption{The two nested separation systems of Example~\ref{ex:td}, and their common structure tree}
   \label{fig:td}\vskip-12pt\vskip0pt
   \end{figure}

\begin{ex}\label{ex:td}
Consider the vertex sets $X_1,\dots,X_4$ shown on the left of Figure~\ref{fig:td}. Let $A$ be a superset of $X_1\cup X_2$ and $B$ a superset of $X_3\cup X_4$, so that $A\cap B\not\sub X_1\cup\dots\cup X_4$ and different $X_i$ do not meet outside~$A\cap B$. Let $\cN$ consist of \AB, \BA, and $(X_1,Y_1),\dots,(X_4,Y_4)$ and their inverses $(Y_i,X_i)$, where $Y_i := (A\cap B)\cup\bigcup_{j\ne i} X_j$. The structure tree $\cT = \T$ has four block nodes $\cX_1,\dots,\cX_4$, with $t_i = [(X_i,Y_i)]$ and $V_{\cX_i} = X_i$, and two central hub nodes
 $$a = \{\AB, (Y_1,X_1), (Y_2,X_2)\}\quad\text{and}\quad b = \{\BA, (Y_3,X_3), (Y_4,X_4)\}$$
joined by the edge~$\{\AB,\BA\}$. The hubs corresponding to $a$ and~$b$ coincide: they are $V_a = A\cap B = V_b$, which is also a block.

Let us now modify this example by enlarging $X_1$ and $X_2$ so that they meet outside $A\cap B$ and each contain $A\cap B$. Thus, $A = X_1\cup X_2$. Let us also shrink $B$ a little, down to $B = X_3\cup X_4$ (Fig.~\ref{fig:td}, right). The structure tree \cT\ remains unchanged by these modifications, but the corresponding sets $V_t$ have changed:
$$V_b = A\cap B\ \subsetneq\ X_1\cap X_2 = X_1\cap Y_1 = X_2\cap Y_2 = V_a,$$
and neither of them is a block, because both are properly contained in~$X_1$, which is also \cN-inseparable.
\end{ex}

Our next lemma shows that deleting a separation from our nested system~\cN\ corresponds to contracting an edge in the structure tree~\T. For a separation $\AB$ that belongs to different systems, we write  $[\AB]_{\cN}$ to indicate in which system \cN\ we are taking the equivalence class.

\begin{lem}\label{lem_extension}
Given $\AB\in\cN$, the tree $\cT':= \cT(\cN')$ for
 $$\cN' = \cN \sm \{\AB,\BA\}$$
arises from $\cT =\T$ by contracting the edge $e = \{\AB,\BA\}$.
  The contracted node $z$ of~$\cT'$ satisfies $z=x\cup y\sm e$ and $V_{z} = V_{x}\cup\, V_{y}$, where $x = [\AB]_{\cN}$ and $y = [\BA]_{\cN}$, and $V(\cT') \sm \{z\} = V(\cT) \sm \{x,y\}$.%
   \COMMENT{}%
   \footnote{The last identity says more than that there exists a canonical bijection between $V(\cT') \sm \{z\}$ and $V(\cT) \sm \{x,y\}$: it says that the nodes of $\cT-\{x,y\}$ and~$\cT'-z$ are the same also as sets of separations.}%
   \COMMENT{}
\end{lem}

\begin{proof}
To see that $V(\cT') \sm \{z\} = V(\cT) \sm \{x,y\}$ and $z=x\cup y\sm e$, we have to show for all $\CD\in\cN'$ that $[\CD]_{\cN}=[\CD]_{\cN'}$ unless $[\CD]_{\cN}\in\{x,y\}$, in which case $[\CD]_{\cN'}=x\cup y\sm e$. In other words, we have to show: 
$$\begin{minipage}[c]{0.85\textwidth}\it
Two separations $\CD, \EF\in\cN'$ are equivalent in $\cN'$ if and only if  they are equivalent in \cN\ or are both in $x\cup y\sm e$.
\end{minipage}\eqno{(\diamond)}$$
Our further claim that $\cT'=\cT/e$, i.e.\ that the node-edge incidences in $\cT'$ arise from those in~$\cT$ as defined for graph minors, will follow immediately from the definition of these incidences in $\cT$ and~$\cT'$.%
   \COMMENT{}
 
Let us prove the backward implication of~$(\diamond)$ first.
As $\cN'\sub\cN$, predecessors in $(\cN,\le)$ are still predecessors in~$\cN'$, and hence $\CD\sim_\cN\EF$ implies $\CD\sim_{\cN'}\EF$. Moreover if $\CD\in x$ and $\EF\in y$ then, in~\cN, \DC~is~a predecessor of \AB\ and \AB\ is a predecessor of \EF. In~$\cN'$, then, \DC\ is a predecessor of \EF, since by Lemma~\ref{path_in_T} and Theorem~\ref{thm_tree} there is no separation $(A',B')\ne\AB$ in~$\cN$ that is both a successor of \DC\ and a predecessor of~\EF. Hence $\CD\sim_{\cN'}\EF$.

For the forward implication in~$(\diamond)$ note that if \DC\ is a predecessor of \EF\ in $\cN'$ but not in~\cN, then in \cN\ we have a sequence of predecessors $\DC<\AB<\EF$ or  $\DC<\BA<\EF$. Then one of \CD\ and \EF\ lies in $x$ and the other in~$y$, as desired.

It remains to show that $V_{z} = V_{x}\cup\, V_{y}$. Consider the sets
 $${x'} := x \sm \{\AB\}\quad\text{and}\quad {y'} := y \sm \{\BA\}\,;$$
then $z = y' \cup x'$.%
   \COMMENT{}
   Since all $\EF \in {x'}$ are equivalent to $\AB$ but not equal to it, we have $\BA \le \EF$ for all those separations. That is,
\begin{equation}\label{bincluded}
B\ \sub \bigcap_{\EF\, \in\, {x'}}\!E\ = V_{x'}.
\end{equation}
By definition of $V_{x}$ we have $V_{x} = V_{x'} \cap A$. Hence \eqref{bincluded} yields ${V_{x'} = V_{x} \cup\, (B \sm A)}$, and since $\sep AB \sub V_{x}$ by Lemma~\ref{nosmallhubs}, we have $V_{x'} = V_{x} \cup B$. An analogous argument yields
 $$V_{y'}= \bigcap_{\EF\, \in\, {y'}}\!E\ =\ V_{y} \cup A.$$
Hence,
\vskip-18pt\begin{eqnarray*}
V_{z} &=& \bigcap_{\EF\, \in\, z}E\\
&=&V_{x'} \cap V_{y'} \\
&=& (V_{x} \cup B) \cap (V_{y} \cup A)\\
&=& (V_{x} \cap V_{y}) \cup (V_{x} \cap A) \cup (V_{y} \cap B) \cup (B \cap A)\\
&=& (V_{x} \cap V_{y}) \cup V_{x}\cup V_{y} \cup (B \cap A)\\
&=&V_{x} \cup V_{y}.%
   \COMMENT{}
\end{eqnarray*}\vskip-18pt
\end{proof}

Every edge $e$ of $\cT$ separates \cT\ into two components. The vertex sets $V_t$ for the nodes $t$ in these components induce a corresponding separation of~$G$, as in \cite[Lemma 12.3.1]{DiestelBook10noEE}. This is the separation that defined~$e$:

\begin{lem}\label{lem_treeedge_sep}
Given any separation $\AB\in\cN$, consider the corresponding edge $e = \{\AB,\BA\}$ of $\cT=\T$. Let $\cT_A$ denote the component of ${\cT-e}$ that contains the node~$[\AB]$, and let $\cT_B$ be the other component. Then ${\bigcup_{t\in\cT_A} V_t = A}$ and $\bigcup_{t\in\cT_B} V_t = B$.
\end{lem}

\begin{proof}
We apply induction on~$|E(\cT)|$. If \cT\ consists of a single edge, the assertion is immediate from the definition of~\cT. Assume now that $|E(\cT)|>1$. In particular, there is an edge $e^*=xy\neq e$. 

Consider $\cN':= \cN \sm e^*$, and let $\cT':= \cT(\cN')$. Then ${\cT' = \cT/e^*}$, by Lemma~\ref{lem_extension}. Let $z$ be the node of~$\cT'$ contracted from~$e^*$. 
Define $\cT'_A$ as the component of $\cT'-e$ that contains the node~$[\AB]$, and let $\cT'_B$ be the other component. We may assume $e^* \in \cT_A$. 
Then
\begin{equation*}
V(\cT_A) \sm \{x , y\} = V(\cT'_A) \sm \{z\}\text{ and }V(\cT_B) = V(\cT'_B).
\end{equation*}
As $V_z=V_x\cup V_y$ by Lemma~\ref{lem_extension},%
   \COMMENT{}
   we can use the induction hypothesis to deduce that
\begin{equation*}
     \bigcup_{t\in\cT_A} V_t =\bigcup_{t\in\cT'_A} V_t = A\quad\text{and}\quad  \bigcup_{t\in\cT_B} V_t =\bigcup_{t\in\cT'_B} V_t = B,
\end{equation*}
as claimed.
\end{proof}

Let us summarize some of our findings from this section. Recall that \cN\ is an arbitrary nested separation system of an arbitrary finite graph~$G$. Let $\cT := \cT(\cN)$ be the structure tree associated with \cN\ as in Section~\ref{sec_tree}, and let $\cV := (V_t)_{t\in\cT}$ be defined by~\eqref{VtDef}. Let us call the separations of $G$ that correspond as in \cite[Lemma 12.3.1]{DiestelBook10noEE} to the edges of the decomposition tree of a \td\ of $G$ the separations {\em induced by\/} this \td.

\begin{thm}\label{treedec}
The pair $\left(\cT,\cV\right)$ is a tree-decomposition of $G$.
\begin{enumerate}[\rm (i)]
\item Every \cN-block is a part of the decomposition.
\item Every part of the decomposition is either an \cN-block or a hub.
\item The separations of $G$ induced by the decomposition are precisely those in~\cN.
\item  Every $\cN'\! \sub \cN$ satsfies $(\cT',\cV') \minor (\cT,\cV)$ for $\cT'\! = \cT(\cN')$ and ${\cV'\!= V(\cT')}$.%
   \footnote{See the Introduction for the definition of $(\cT',\cV') \minor (\cT,\cV)$.}
\end{enumerate}
\end{thm}

\begin{proof} Of the three axioms for a \td, (T1)~and (T2) follow from Lemma~\ref{blocksarenodes}, because single vertices and edges form \cN-inseparable vertex sets, which extend to \cN-blocks. For the proof of~(T3), let $e=\{\AB,\BA\}$ be an edge at~$t_2$ on the ${t_1}$--${t_3}$ path in~\cT. Since $e$ separates ${t_1}$ from ${t_3}$ in~\cT, Lemmas \ref{lem_treeedge_sep} and~\ref{nosmallhubs} imply that $V_{t_1}\cap V_{t_3} \sub A\cap B \sub V_{t_2}$.

Statement~(i) is Lemma~\ref{blocksarenodes}. Assertion~(ii) is Lemma~\ref{lem_eq_block}. Assertion~(iii) follows from Lemma~\ref{lem_treeedge_sep} and the definition of the edges of~\cT. Statement~(iv) follows by repeated application of Lemma~\ref{lem_extension}.
\end{proof}

\section{Extracting nested separation systems}\label{sec_sys}

Our aim in this section will be to find inside a given \sys~\cS\ a nested subsystem~\cN\ that can still distinguish the elements of some given set \cI\ of \cS-inseparable sets of vertices. As we saw in Sections \ref{sec_tree} and~\ref{sec_td}, such a nested subsystem will then define a \td\ of~$G$, and the sets from~\cI\ will come to lie in different parts of that decomposition.

This cannot be done for all choices of \cS\ and~\cI. Indeed, consider the following example of where such a nested subsystem does not exist. Let $G$ be the $3\times 3$-grid, let \cS\ consist of the two 3-separations cutting along the horizontal and the vertical symmetry axis, and let \cI\ consist of the four corners of the resulting cross-diagram. Each of these is \cS-inseparable, and any two of them can be separated by a separation in~\cS. But since the two separations in \cS\ cross, any nested subsystem contains at most one of them, and thus fails to separate some sets from~\cI.

However, we shall prove that the desired nested subsystem does exist if \cS\ and \cI\ satisfy the following condition. Given a \sys~\cS\ and a set~\cI\ of \cS-inseparable sets, let us say that \emph{\cS\ separates \cI\ well} if the following holds for every pair of crossing~-- that is, not nested~-- separations $\AB,\CD\in\cS$:
   \begin{txteq*}
For all $I_1,I_2 \in \cI$ with $I_1\sub \sep AC$ and $I_2\sub\sep BD$ there is an $\EF\in\cS$ such that $I_1\sub E\sub\sep AC$ and $F\supseteq B\cup D$.%
   \COMMENT{}
   \end{txteq*}
Note that such a separation satisfies both $\EF\le\AB$ and $\EF\le\CD$.

In our grid example, \cS\ did not separate \cI\ well, but we can mend this by adding to \cS\ the four corner separations. And as soon as we do that, there is a nested subsystem that separates all four corners~-- for example, the set of the four corner separations.

More abstractly, the idea behind the notion of \cS\ separating \cI\ well is as follows. In the process of extracting \cN\ from~\cS\ we may be faced with a pair of crossing separations \AB\ and \CD\ in~\cS\ that both separate two given sets $I_1,I_2\in\cI$, and wonder which of them to pick for~\cN. (Obviously we cannot choose both.) If \cS\ separates \cI\ well, however, we can avoid this dilemma by choosing~\EF\ instead: this also separates $I_1$ from~$I_2$, and since it is nested with both \AB\ and \CD\ it will not prevent us from choosing either of these later too, if desired.

Let us call a separation ${\EF\in\cS}$ \emph{extremal} in \cS\ if for all $\CD\in\cS$ we have either $\EF\le\CD$ or $\EF\le\DC$. In particular, extremal separations are nested with all other separations in~\cS. Being extremal implies being $\le$-minimal in~\cS;%
   \COMMENT{}
   if \cS\ is nested, extremality and $\leq$-minimality are equivalent. If $\EF\in\cS$ is extremal, then $E$ is an \cS-block;%
   \COMMENT{}
   we call it an \emph{extremal block} in~\cS.%
   \COMMENT{}

A separation system, even a nested one, typically contains many extremal separations. For example, given a \td\ of $G$ with decomposition tree~\cT, the separations corresponding to the edges of \cT\ that are incident with a leaf of~\cT\ are extremal in the (nested) set of all the separations of $G$ corresponding to edges of~\cT.%
   \footnote{More precisely, every such edge of \cT\ corresponds to an inverse pair of separations of which, usually, only one is extremal: the separation \AB\ for which $A$ is the part $V_t$ with $t$ a leaf of~\cT. The separation \BA\ will not be extremal, unless $\cT=K^2$.}

Our next lemma shows that separating a set \cI\ of \cS-inseparable sets well is enough to guarantee the existence of an extremal separation among those that separate sets from \cI. Call a separation \emph{\cI-relevant} if it weakly separates some two sets in~\cI. If all the separations in \cS\ are \cI-relevant, we call \cS\ itself \emph{\cI-relevant}.

\begin{lem}\label{lem_extremal}
Let $\cR$ be a \sys\ that is \cI-relevant for some set \cI\ of \cR-inseparable sets. If \cR\ separates \cI\ well, then every $\le$-minimal $\AB\in\cR$ is extremal in~\cR. In particular, if $\cR\ne\es$ then \cR\ contains an extremal separation.
\end{lem}

\begin{proof}
Consider a $\le$-minimal separation $\AB\in\cR$, and let $\CD\in\cR$ be given. If \AB\ and \CD\ are nested, then the minimality of~$\AB$ implies that $\AB\leq \CD$ or $\AB\leq \DC$,
   \COMMENT{}
  as desired. So let us assume that \AB\ and \CD\ cross.

As \AB\ and \CD\ are \cI-relevant and the sets in \cI\ are \cR-inseparable, we can find opposite corners of the cross-diagram $\{\AB,\CD\}$ that each contains a set from~\cI. Renaming \CD\ as \DC\ if necessary, we may assume that these sets lie in \sep AC\ and \sep BD, say $I_1\sub\sep AC$ and $I_2\sub\sep BD$. As \cR\ separates \cI\ well, there exists $\EF\in\cR$ such that $I_1 \sub E \sub\sep AC$ and $F\supseteq B\cup D$, and hence $\EF\le\AB$ as well as $\EF\le\CD$. By the minimality of \AB, this yields $\AB=\EF\le\CD$ as desired.
\end{proof}

Let us say that a set \cS\ of separations \emph{distinguishes} two given \cS-inseparable sets $I_1,I_2$ (or \emph{distinguishes} them {\em properly\/}) if it contains a separation that separates them. If it contains a separation that separates them weakly, it \emph{weakly distinguishes\/} $I_1$ from~$I_2$. We then also call $I_1$ and $I_2$ \emph{(weakly) distinguishable} by~\cS, or \emph{(weakly) \cS-distinguishable}.

Here is our main result for this section:

\begin{thm}\label{thm_main}
Let \cS\ be any \sys\ that separates some set \cI\ of \cS-inseparable sets of vertices well. Then \cS\ has a nested \cI-relevant subsystem $\cN(\cS,\cI) \sub \cS$ that weakly distinguishes all weakly \cS-distinguishable sets in~\cI.
\end{thm}

\begin{proof}
If \cI\ has no two weakly distinguishable elements, let $\cN(\cS,\cI)$ be empty. Otherwise let $\cR\sub\cS$ be the subsystem of all \cI-relevant separations in~\cS. Then $\cR\ne\es$, and \cR~separates \cI\ well.%
   \COMMENT{}
   Let $\cE\sub\cR$ be the subset%
   \COMMENT{}
   of those separations that are extremal in \cR, and put
 $$\closure\cE := \{ \AB\;|\;\AB \mbox{ or }\BA \mbox{ is in }\cE \}.$$
By Lemma \ref{lem_extremal} we have $\closure\cE\neq\emptyset$, and by definition of extremality all separations in $\closure\cE$ are nested with all separations in~\cR. In particular, $\closure\cE$ is nested.

Let
 $$\cI_\cE := \{I\in\cI\;|\;\exists\EF\in\cE:I\sub E\}.$$
This is non-empty, since $\cE \sub \cR$ is non-empty and \cI-relevant. Let us prove that \cE\ weakly distinguishes all pairs of weakly distinguishable elements $I_1, I_2\in\cI$ with $I_1\in\cI_\cE$. Pick $\AB\in\cR$%
   \COMMENT{}
   with $I_1\sub A$ and $I_2\sub B$. Since $I_1\in\cI_\cE$, there is an $\EF\in\cE$ such that $I_1 \sub E$. By the extremality of \EF\ we have either $\EF\le\AB$, in which case $I_1 \sub E$ and $I_2 \sub B\sub F$, or we have $\EF\le\BA$, in which case $I_1 \sub E\cap A\sub \sep EF$. In both cases $I_1$ and $I_2$ are weakly separated by \EF.%
   \COMMENT{}

As $\cI' := \cI\sm\cI_\cE$ is a set of \cS-inseparable sets with fewer elements than~\cI,%
   \COMMENT{}
   induction gives us a nested $\cI'$-relevant subsystem $\cN(\cS,\cI')$ of~\cS\ that weakly distinguishes all weakly distinguishable elements of~$\cI'$. Then
 $$\cN(\cS,\cI) := \closure\cE \cup \cN(\cS,\cI')$$
 is \cI-relevant and weakly distinguishes all weakly distinguishable elements of~\cI.
As $\cI'\sub\cI$, and thus $\cN(\cS,\cI')\sub\cR$, the separations in $\closure\cE$ are nested with those in~$\cN(\cS,\cI')$.%
   \COMMENT{}
   Hence, $\cN(\cS,\cI)$ too is nested.
\end{proof}

An important feature of the proof of Theorem~\ref{thm_main} is that the subset $\cN(\cS,\cI)$ it constructs is \emph{canonical}, given \cS\ and~\cI: there are no choices made anywhere in the proof. We may thus think of $\cN$ as a recursively defined operator that assigns to every pair $(\cS,\cI)$ as given in the theorem a certain nested subsystem $\cN(\cS,\cI)$ of~$\cS$. This subsystem $\cN(\cS,\cI)$ is canonical also in the structural sense that it is invariant under any automorphisms of~$G$ that leave \cS\ and \cI\ invariant.

To make this more precise, we need some notation. Every automorphism $\alpha$ of $G$ acts also on (the set of) its vertex sets $U\sub V(G)$, on the collections $\cal X$ of such vertex sets, on the separations \AB\ of~$G$, and on the sets \cS\ of such separations. We write $U^\alpha$, ${\cal X}^\alpha$, $\AB^\alpha$ and $\cS^\alpha$ and so on for their images under~$\alpha$.

\begin{cor}\label{cor_properties}
Let \cS\ and \cI\ be as in Theorem \ref{thm_main}, and let $\cN(\cS,\cI)$ be the nested subsystem of~\cS\ constructed in the proof. Then for every automorphism $\alpha$ of~$G$ we have $\cN(\cS^\alpha,\cI^\alpha)= \cN(\cS,\cI)^\alpha$. In particular, if \cS\ and \cI\ are invariant under the action of a group $\Gamma$ of automorphisms of~$G$,%
   \COMMENT{}
    then so is $\cN(\cS,\cI)$.
\end{cor}

\begin{proof}
The proof of the first assertion is immediate from the construction of $\cN(\cS,\cI)$ from $\cS$ and~$\cI$. The second assertion follows, as
 $$\cN(\cS,\cI)^\alpha = \cN(\cS^\alpha,\cI^\alpha) = \cN(\cS,\cI)$$
for every $\alpha\in\Gamma$.%
   \COMMENT{}
\end{proof}

\section{\boldmath Separating the $k$-blocks of a graph}\label{sec_kinsep}
We now apply the theory developed in the previous sections to our original problem, of how to `decompose a graph $G$ into its $(k+1)$-connected components'. In the language of Section~\ref{sec_sys}, we consider as \cS\ the set of all proper $k$-separations of~$G$, and as \cI\ the set of its $k$-blocks. Our results from Section~\ref{sec_sys} rest on the assumption that the set~\cR\ of \cI-relevant separations in \cS\ separates~\cI\ well (Lemma~\ref{lem_extremal}). So the first thing we have to ask is: given crossing $k$-separations \AB\ and~\CD\ such that $A\cap C$ and $B\cap D$ contain $k$-blocks $b_1$ and~$b_2$, respectively, is there a $k$-separation $\EF$ such that $b_1\sub E\sub\sep AC$?

If $G$ is $k$-connected, there clearly is. Indeed, as the corners $A\cap C$ and $B\cap D$ each contain a $k$-block, they have order at least~$k+1$, so their boundaries cannot have size less than~$k$.%
   \COMMENT{}
    But the sizes of these two corner boundaries sum to ${|A\cap B| + |C\cap D| = 2k}$, so they are both exactly~$k$. We can thus take as \EF\ the corner separation ${(A\cap C}, {B\cup D})$.

If $G$ is not $k$-connected, we shall need another reason for these corner separations to have order at least~$k$. This is a non-trivial problem. Our solution will be to assume inductively that those $k$-blocks that can be separated by a separation of \text{order} $\ell < k$ are already separated by such a separation selected earlier in the induction. Then the two corner separations considered above will have order at least~$k$, since the $k$-blocks in the two corners are assumed not to have been separated earlier.

This approach differs only slightly from the more ambitious approach to build, inductively on~$\ell$, one nested set of separations which, for all $\ell$ at once, distinguishes every two $\ell$-blocks by a separation of order at most~$\ell$. We shall construct an example showing that such a unified nested separation system need not exist. The subtle difference between our approach and this seemingly more natural generalization is that we use $\ell$-separations for $\ell < k$ only with the aim to separate $k$-blocks; we do not aspire to separate all $\ell$-blocks, including those that contain no $k$-block.

However we shall be able to prove that the above example is essentially the only one precluding the existence of a unified nested set of separations. Under a mild additional assumption saying that all blocks considered must be `robust', we shall obtain one unified nested set of separations that distinguishes, for all $\ell$ simultaneously, all $\ell$-blocks by a separation of order at most~$\ell$. All $\ell$-blocks that have size at least ${3\over2}\ell$ will be robust.

Once we have found our nested separation systems, we shall convert them into \td s as in Section~\ref{sec_td}. Both our separation systems and our \td s will be canonical in that they depend only on the structure of~$G$. In particular,%
   \COMMENT{}
   they will be invariant under the automorphism group $\Aut(G)$ of~$G$.

   \begin{figure}[htpb]
\begin{center}
   	  \includegraphics[width=5cm]{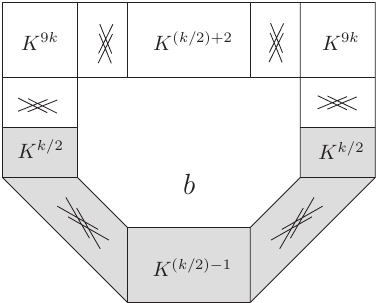}
      \captionsetup{margin=10pt,font=small,labelfont=it}
   	  \caption{A horizontal $k$-separation needed to distinguish two $k$-blocks,\penalty-200\ crossed by a vertical
 $(k+1)$-separation needed to distinguish two $(k+1)$-blocks.}
   	  \label{fig:robust_needed}\vskip-12pt\vskip0pt
\end{center}
   \end{figure}

Let us now turn to our example showing that a graph need not have a `unified' nested separation system \cN\ of separations of mixed order that distinguishes, for every~$\ell$, distinct $\ell$-blocks by a separation in \cN\ of order at most~$\ell$. The graph depicted in Figure~\ref{fig:robust_needed} arises from the disjoint union of a $K^{(k/2)-1}$, two $K^{k/2}$, a $K^{(k/2)+2}$ and two $K^{9k}$, by joining the $K^{(k/2)-1}$ completely to the two $K^{k/2}$, the $K^{(k/2)+2}$ completely to the two $K^{9k}$, the left $K^{k/2}$ completely to the left $K^{9k}$, and the right $K^{k/2}$ completely to the right $K^{9k}$. 
The horizontal $k$-separator consisting of the two $K^{k/2}$ defines the only 
separation of order at most~$k$ that distinguishes the two $k$-blocks consisting of the top five complete 
graphs versus the bottom three. On the other hand, the vertical $(k+1)$-separator consisting of the $K^{(k/2)-1}$ and the $K^{(k/2)+2}$ defines the only separation of order at most $(k+1)$ that distinguishes the two $(k+1)$-blocks consisting, respectively, of
the left $K^{k/2}$ and $K^{9k}$ and the $K^{(k/2)+2}$, and of the right $K^{k/2}$ and $K^{9k}$ and the $K^{(k/2)+2}$.
Hence any separation system that distinguishes all $k$-blocks as well 
as all $(k+1)$-blocks must contain both separations. Since the two separations cross, such a system cannot be nested.

In view of this example it may be surprising that we can find a separation system that distinguishes,%
   \COMMENT{}
   for all $\ell\ge 0$ simultaneously, all \emph{large} $\ell$-blocks of~$G$, those with at least $\lfloor{3\over 2}\ell\rfloor$ vertices. The example of Figure~\ref{fig:robust_needed} shows that this value is best possible:
here, all blocks are large except for the $k$-block $b$ consisting of the two $K^{k/2}$ and the~$K^{(k/2)-1}$, which has size~${3\over 2}k - 1$.%
 \COMMENT{}

Indeed, we shall prove something considerably stronger: that the only obstruction to the existence of a unified \td\ is a $k$-block that is not only not large but positioned exactly like $b$ in Figure~\ref{fig:robust_needed}, inside the union of a $k$-separator and a larger separator crossing it.

Given integers $k$ and~$K$ (where $k\le K$ is the interesting case, but it is important formally to allow $k>K$),%
   \COMMENT{}
   a $k$-inseparable set $U$ is called {\em $K$-robust}%
   \footnote{The parameter $k$ is important here, too, but we suppress it for readability; it will always be stated explicitly in the context.}
   if for every $k$-separation \CD\ with $U\sub D$ and every separation \AB\ of order at most $K$ such that $\AB\nparallel \CD$ and
\begin{equation}\label{robustbdries}
   |\partial (A\cap D)| <k>|\partial (B\cap D)|\,,
\end{equation}
we have either $U\sub A$ or $U\sub B$. By $U\sub D$ and~\eqref{robustbdries}, the only way in which this can fail is that $|A\cap B|>k$ and $U$ is contained in the union $T$ of the boundaries of $A\cap D$ and $B\cap D$ (Fig.~\ref{fig:robust}):%
   \COMMENT{}
  exactly the situation of $b$ in~Figure~\ref{fig:robust_needed}.

   \begin{figure}[htpb]
\begin{center}
   	  \includegraphics[width=4cm]{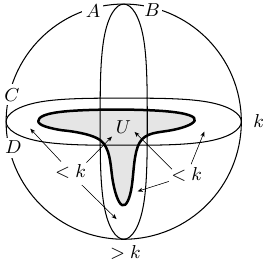}
   	  \caption{The shaded set $U$ is $k$-inseparable but not $K$-robust.}
   	  \label{fig:robust}\vskip-18pt\vskip0pt
\end{center}
   \end{figure}

It is obvious from the definition of robustness that
\begin{equation}\label{KinsepKrobust}
\text{\em for $k\ge K$, every $k$-inseparable set is $K$-robust.}
\end{equation}

Let us call a $k$-inseparable set, in particular a $k$-block of~$G$, {\em robust\/} if it is $K$-robust for every~$K$ (equivalently, for $K=|G|$). Our next lemma says that large $k$-blocks, those of size at least~$\lfloor{3\over 2}k\rfloor$, are robust. But there are more kinds of robust sets than these: the vertex set of any $K^{k+1}$ subgraph, for example, is a robust $k$-inseparable set.%
   \COMMENT{}

\begin{lem}\label{robusti}
Large $k$-blocks are robust.
\end{lem}

\begin{proof}
By the remark following the definition of `$K$-robust', it suffices to show that the set $T = \partial(A\cap D)\cup \partial (B\cap D)$ in Figure~\ref{fig:robust} has size at most~$\frac{3}{2}k-1$, regardless of the order of $\AB$. Let $\ell := |(A\cap B)\sm C|$ be the size of the common link of the corners $A\cap D$ and $B\cap D$. By $|C\cap D| = k$ and~\eqref{robustbdries} we have $2\ell\le k-2$,%
   \COMMENT{}
   so $|T| = k+\ell\le {3\over2}k-1$ as desired.
\end{proof}

For the remainder of this paper, a {\em block\/} of $G$ is again a subset of $V(G)$ that is a $k$-block for some~$k$.
The smallest $k$ for which a block $b$ is a $k$-block is its {\em rank\/}; let us denote this by~$r(b)$.
A block $b$ that is given without a specified~$k$ is called {\em $K$-robust} if it is $K$-robust as an $r(b)$-inseparable set. When we speak of a `robust $k$-block'~$b$, however, we mean the (stronger, see below) robustness as a $k$-inseparable set, not just as an $r(b)$-inseparable set.

   It is not difficult to find examples of $K$-robust blocks that are $k$-blocks but are not $K$-robust as a $k$-block,%
   \COMMENT{}
   only as an $\ell$-block for some $\ell<k$.%
   \COMMENT{}
   A~$k$-inseparable set that is $K$-robust as a $k'$-inseparable set for $k'>k$, however, is also $K$-robust as a $k$-inseparable set. More generally:

\begin{lem}\label{robustii} Let $k$, $k'$ and $K$ be integers. 
\begin{enumerate}[\rm (i)]
\item Every $k$-inseparable set $I$ containing a $K$-robust $k'$-inseparable set~$I'$ with $k\le k'$ is $K$-robust.
\item Every block $b$ that contains a $K$-robust block $b'$ is $K$-robust.
\end{enumerate}
\end{lem}

\begin{proof}
(i) Suppose that $I$ is not $K$-robust, and let this be witnessed by a $k$-separation \CD\ crossed by a separation \AB\ of order $m\le K$. Put $S:= C\cap D$ and $L:= (A\cap B)\sm C$. 
Then $I\sub S\cup L$, as remarked after the definition of `$K$-robust'.

Extend $S$ into $L$ to a $k'$-set~$S'$ that is properly%
   \COMMENT{}
   contained in $S\cup L$ (which is large enough, since it contains $I'\sub I$), and put $C':= C\cup S'$. 
Then $(C',D)$ is a $k'$-separation with separator~$S'$ and corners $D\cap A$ and $D\cap B$ with~\AB, whose boundaries by assumption have size less than~$k\le k'$. 
As $I'$ is $K$-robust, it lies in one of these corners, say  $I'\sub A\cap D$. Since
  $$|I'| > k' \ge k > |\partial (A\cap D)|\,,$$
this implies that~$I'$ has a vertex in the interior of the corner $A\cap D$. As $I'\sub I$, this contradicts the fact that $I\sub S\cup L$.%
   \COMMENT{}

(ii) The block $b$ is an $r(b)$-inseparable set containing the $K$-robust $r(b')$-inseparable set $b'$. If $b = b'$ then $r(b) = r(b')$. If $b\supsetneq b'$, then $b'$ is not maximal as an $\ell$-inseparable set for any $\ell\le r(b)$, giving $r(b') > r(b)$. Hence $r(b) \le r(b')$ either way, so $b$ is a $K$-robust block by~(i).
\end{proof}

Let us call two blocks {\em distinguishable\/} if neither contains the other. It is not hard to show that distinguishable blocks $b_1, b_2$ can be separated in $G$ by a separation of order $r\le \min\{r(b_1),r(b_2)\}$.%
   \COMMENT{}
   We denote the smallest such $r$ by
 $$\kappa(b_1,b_2)\le \min\{r(b_1),r(b_2)\},$$
and say that $b_1$ and $b_2$ are {\em $k$-distinguishable\/} for a given integer~$k$ if $\kappa(b_1,b_2)\le k$. Note that distinct $k$-blocks are $k$-distinguishable, but they might also be $\ell$-distin\-guish\-able for some $\ell < k$.

A~set \cS\ of separations {\em distinguishes\/} two $k$-blocks if it contains a separation of order at most~$k$ that separates them. It {\em distinguishes\/} two blocks%
   \COMMENT{}
  $b_1, b_2$ given without a specified~$k$ if it contains a separation of order $r\le\min\{r(b_1),r(b_2)\}$ that separates them.%
   \footnote{Unlike in the definition just before Theorem~\ref{thm_main}, we no longer require that the blocks we wish to separate be \cS-inseparable for the entire set~\cS.}%
   \COMMENT{}
   If \cS\ contains a separation of order $\kappa(b_1,b_2)$ that separates two blocks or $k$-blocks $b_1,b_2$, we say that \cS\ distinguishes them {\em efficiently\/}.

\begin{thm}\label{thm_kinsep_sys}
For every finite graph $G$ and every integer $k\ge 0$ there is a tight, nested, and $\Aut(G)$-invariant \sys\ $\cN_k$ that distinguishes every two $k$-distinguishable $k$-robust blocks efficiently. In particular,%
   \COMMENT{}
   $\cN_k$~distinguishes every two $k$-blocks efficiently.
   \end{thm}

\begin{proof} Let us rename the integer $k$ given in the theorem as~$K$. Recursively for all integers $0 \le k \le K$ we shall construct a sequence of \sys s $\cN_k$ with the following properties:

\begin{enumerate}[(i)]
\item $\cN_k$ is tight, nested, and $\Aut(G)$-invariant;\label{props}
\item $\cN_{k-1}\sub\cN_k$ (put $\cN_{-1} := \es$);\label{inclusion}
\item every separation in $\cN_k\sm \cN_{k-1}$ has order~$k$;\label{order}
\item $\cN_k$ distinguishes every two $K$-robust $k$-blocks.\label{disting}%
   \COMMENT{}
\item every separation in $\cN_k\sm \cN_{k-1}$ separates some $K$-robust $k$-blocks that are not distinguished by~$\cN_{k-1}$.\label{lean}%
   \COMMENT{}
   \end{enumerate}
We claim that $\cN_K$ will satisfy the assertions of the theorem for $k=K$. Indeed, consider two $K$-distinguishable $K$-robust blocks $b_1,b_2$. Then
 $$\kappa:= \kappa(b_1,b_2)\le \min\{K,r(b_1),r(b_2)\},$$
so $b_1,b_2$ are $\kappa$-inseparable and extend to distinct $\kappa$-blocks $b'_1,b'_2$. These are again $K$-robust, by Lemma~\ref{robustii}\,(i).%
   \COMMENT{}
   Hence by \eqref{disting}, $\cN_\kappa\sub\cN_K$ distinguishes $b'_1\supseteq b_1$ from $b'_2\supseteq b_2$, and it does so efficiently by definition of~$\kappa$.

   \COMMENT{}

It remains to construct the separation systems~$\cN_k$.

Let $k\ge 0$ be given, and assume inductively that we already have \sys s~$\cN_{k'}$ satisfying \eqref{props}--\eqref{lean} for $k' = 0,\ldots, k-1$. (For $k=0$ we have nothing but the definiton of $\cN_{-1} := \es$, which has $V(G)$ as its unique $\cN_{-1}$-block.) Let us show the following:
\begin{txteq}\label{ind}\rm
For all $0\le\ell\le k$, any two $K$-robust $\ell$-blocks $b_1,b_2$ that are not distinguished by~$\cN_{\ell-1}$ satisfy $\kappa(b_1,b_2)=\ell$.%
   \COMMENT{}
\end{txteq}
This is trivial for $\ell=0$; let $\ell>0$. If $\kappa(b_1,b_2) < \ell$, then the $(\ell-1)$-blocks $b'_1\supseteq b_1$ and $b'_2\supseteq b_2$ are distinct. By Lemma~\ref{robustii}\,(i) they are again $K$-robust. Thus by hypothesis~\eqref{disting} they are distinguished by~$\cN_{\ell-1}$, and hence so are $b_1$ and~$b_2$, contrary to assumption.

By hypothesis \eqref{order},%
   \COMMENT{}
   every $k$-block is $\cN_{k-1}$-inseparable, so it extends to some $\cN_{k-1}$-block; let $\cB$ denote the set of those $\cN_{k-1}$-blocks that contain more than one $K$-robust $k$-block.%
   \COMMENT{}
   For each $b \in \cB$ let $\cI_b$ be the set of all $K$-robust $k$-blocks contained in~$b$. 
Let $\cS_b$ denote the set of all those $k$-separations of~$G$ that separate some two elements of $\cI_b$ and are nested with all the separations in~$\cN_{k-1}$.

Clearly $\cS_b$ is symmetric and the separations in $\cS_b$ are proper (since they distinguish two $k$-blocks), so $\cS_b$ is a separation system of~$G$. By~\eqref{ind} for $\ell=k$, the separations in~$\cS_b$ are tight. Our aim is to apply Theorem \ref{thm_main} to extract from~$\cS_b$ a nested subsystem $\cN_b$ that we can add to~$\cN_{k-1}$.

Before we verify the premise of Theorem \ref{thm_main}, let us prove that it will be useful: that the nested separation system $\cN_b\sub\cS_b$ it yields can distinguish%
   \footnote{As the elements of $\cI_b$ are $k$-blocks, we have two notions of `distinguish' that could apply: the definition given before Theorem \ref{thm_main}, or that given before Theorem~\ref{thm_kinsep_sys}. However, as $\cS_b$ consists of $k$-separations and all the elements of $\cI_b$ are $\cS_b$-inseparable,%
   \COMMENT{}
   the two notions coincide.\looseness=-1}
   all the elements of~$\cI_b$. This will be the case only if $\cS_b$ does so, so let us prove this first:

\claim{($\ast$)}{$\cS_b$ distinguishes every two elements of~$\cI_b$.}
For a proof of $(*)$ we have to find for any two $k$-blocks $I_1,I_2 \in \cI_b$ a separation in~$\cS_b$ that separates them. Applying Lemma~\ref{block_lem} with the set \cS\ of all separations of order at most~$k$, we can find a separation $\AB\in\cS$ such that $I_1 \sub A$ and $I_2 \sub B$. Choose \AB\ so that it is nested with as many separations in $\cN_{k-1}$ as possible. We prove that $\AB\in\cS_b$, by showing that \AB\ has order exactly~$k$ and is nested with every separation $\CD\in\cN_{k-1}$. Let $\CD\in\cN_{k-1}$ be given.

Being elements of~$\cI_b$, the sets $I_1$ and~$I_2$ cannot be separated by fewer than $k$ vertices, by~\eqref{ind}. Hence \AB\ has order exactly~$k$. Since $I_1$ is $k$-inseparable it lies on one side of~\CD, say in~$C$, so $I_1\sub A\cap C$. As \CD\ does not separate $I_1$ from~$I_2$,%
   \COMMENT{}
   we then have $I_2\sub B\cap C$. 

Let $\ell < k$ be such that $\CD\in\cN_\ell\sm\cN_{\ell-1}$. By hypothesis~\eqref{lean} for~$\ell$, 
there are $K$-robust $\ell$-blocks $J_1\sub C$ and $J_2\sub D$ that are not distinguished by~$\cN_{\ell-1}$. By~\eqref{ind},
\begin{equation}\label{notbym}
\kappa(J_1,J_2) = \ell.
\end{equation}
 \COMMENT{}
Let us show that we may assume the following:
\begin{txteq}\label{cornernested}\rm
The corner separations of the corners $A\cap C$ and $B\cap C$ are nested with every separation $(C',D')\in\cN_{k-1}$ that \AB\ is nested with.
\end{txteq}
Since \CD\ and $(C',D')$ are both elements of~$\cN_{k-1}$, they are nested with each other. Thus,
 $$\AB \| (C',D') \| \CD.$$
Unless \AB\ is nested with \CD\ (in which case our proof of~$(*)$ is complete), this implies by Lemma \ref{lem_nested_or} that $(C',D')$ is nested with all the corner separations of the cross-diagram for $(A,B)$ and $(C,D)$, especially with those of the corners $A\cap C$ and $B\cap C$ that contain $I_1$ and~$I_2$. This proves~\eqref{cornernested}.

Since the corner separations of $A\cap C$ and $B\cap C$ are nested with the separation $\CD\in\cN_{k-1}$ that \AB\ is not nested with (as we assume),%
   \COMMENT{}
\eqref{cornernested}~and the choice of $(A,B)$ imply that
$$|\partial (A\cap C)|\geq k+1\quad\text{and}\quad |\partial (B\cap C)|\geq k+1.$$
Since the sizes of the boundaries of two opposite corners sum to
$$|A\cap B| + |C\cap D| = k+\ell,$$
this means that the boundaries of the corners $A\cap D$ and $B\cap D$ have sizes~$<\ell$.
Since $J_2$ is $K$-robust as an $\ell$-block, we thus have $J_2\sub A\cap D$ or $J_2\sub B\cap D$, say the former.
 But as $J_1\sub C\sub B\cup C$, this contradicts~\eqref{notbym}, completing the proof of~$(*)$.

\medbreak

Let us now verify the premise of Theorem \ref{thm_main}:

\claim{($\ast\ast$)}{$\cS_b$ separates $\cI_b$ well.}
Consider a pair $\AB,\CD\in\cS_b$ of crossing separations with sets $I_1,I_2 \in \cI_b$ such that $I_1 \sub \sep AC$ and $I_2 \sub \sep BD$. We shall prove that $\csepn ACBD\in\cS_b$.

By \eqref{ind} and $I_1,I_2 \in \cI_b$, the boundaries of the corners \sep AC\ and \sep BD\ have size at least~$k$. Since their sizes sum to $|A\cap B| + |C\cap D| = 2k$, they each have size exactly~$k$. Hence \csepn ACBD has order~$k$ and is nested with every separation $(C',D')\in\cN_{k-1}$%
   \COMMENT{}
   by Lemma~\ref{lem_nested_or}, because $\AB,\CD\in\cS_b$ implies that \AB\ and~\CD\ are both nested with $(C',D')\in\cN_{k-1}$. This completes the proof of~$(**)$.

\medbreak

By ($\ast$) and~($\ast\ast$), Theorem \ref{thm_main} implies that $\cS_b$ has a nested $\cI_b$-relevant subsystem $\cN_b := \cN(\cS_b,\cI_b)$ that weakly distinguishes all the sets in~$\cI_b$. But these are $k$-inseparable and hence of size~$>k$, so they cannot lie inside a $k$-separator. So $\cN_b$ even distinguishes the sets in~$\cI_b$ properly. Let
 $$\cN_\cB := \bigcup_{b\in\cB}\cN_b\quad\text{and}\quad \cN_k := \cN_{k-1} \cup \cN_\cB.$$

Let us verify the inductive statements \eqref{props}--\eqref{lean} for~$k$. We noted earlier that every $\cS_b$ is tight, hence so is every~$\cN_b$. The separations in each $\cN_b$ are nested with each other and with~$\cN_{k-1}$. Separations from\vadjust{\penalty-99} different sets~$\cN_b$ are nested by Lemma~\ref{lem_block_sep}. So the entire set $\cN_k$ is nested. Since $\cN_{k-1}$ is $\Aut(G)$-invariant, by hypothesis~\eqref{props}, so is~$\cB$. For every automorphism~$\alpha$ and every $b\in\cB$ we then have $\cI_{b^\alpha} = (\cI_b)^\alpha$ and $\cS_{b^\alpha} = (\cS_b)^\alpha$, so Corollary~\ref{cor_properties} yields $(\cN_b)^\alpha = \cN_{b^\alpha}$. Thus, $\cN_\cB$~is $\Aut(G)$-invariant too, completing the proof of~(i).
Assertions \eqref{inclusion} and~\eqref{order} hold by definition of~$\cN_k$. Assertion~\eqref{disting} is easy too: if two $K$-robust $k$-blocks are not distinguished by~$\cN_{k-1}$ they will lie in the same $\cN_{k-1}$-block~$b$, and hence be distinguished by~$\cN_b$. Assertion~\eqref{lean} holds, because each $\cN_b$ is $\cI_b$-relevant.
   \end{proof} 

Let us call two blocks $b_1,b_2$ of~$G$ {\em robust\/} if there exists a $k$ for which they are robust $k$-blocks.%
   \footnote{By Lemma~\ref{robustii}\,(i), this is equivalent to saying that they are robust $r(b_i)$-blocks, that is, $K$-robust $r(b_i)$-blocks for $K=|G|$.}
   For $k=|G|$,%
   \COMMENT{}
   Theorem~\ref{thm_kinsep_sys} then yields our `unified' nested separation system that separates all robust blocks by a separation of the lowest possible order:

\begin{cor}
For every finite graph $G$ there is a tight, nested, and $\Aut(G)$-invariant \sys\ $\cN$ that distinguishes every two distinguishable robust blocks efficiently.\qed
\end{cor}

Let us now turn the separation systems $\cN_k$ of Theorem~\ref{thm_kinsep_sys} and its proof into \td s:

\begin{thm}\label{thm_kinsep_treedec}
For every finite graph $G$ and every integer $K$ there is a sequence $\left(\cT_k,\cV_k\right)_{k \le K}$ of tree-decompositions such that, for all~$k \le K$,
\begin{enumerate}[\rm (i)]
\item every $k$-inseparable set is contained in a unique part of~$(\cT_k,\cV_k)\,;$
\item distinct $K$-robust $k$-blocks lie in different parts of $(\cT_k,\cV_k)\,;$%
   \COMMENT{}
   \item $\left(\cT_k,\cV_k\right)$ has adhesion at most $k\,;$
\item if $k > 0$ then $\left(\cT_{k-1},\cV_{k-1}\right) \minor \left(\cT_k,\cV_k\right);$
\item $\Aut(G)$ acts on $\cT_k$ as a group of automorphisms.
\end{enumerate}
\end{thm}

\begin{proof}
Consider the nested \sys\ $\cN_K$ given by Theorem~\ref{thm_kinsep_sys}. As in the proof of that theorem, let $\cN_k$ be the subsystem of $\cN_K$ consisting of its separations of order at most~$k$. By Theorem~\ref{thm_kinsep_sys}, $\cN_K$ is $\Aut(G)$-invariant, so this is also true for all $\cN_k$ with $k < K$.%
   \COMMENT{}

Let $(\cT_k,\cV_k)$ be the tree-decomposition associated with $\cN_k$ as in Section~\ref{sec_td}. Then (v) holds by Corollary~\ref{cor_tn_Gamma}, (iii)~and (iv) by Theorem~\ref{treedec}\,(iii) and~(iv). By~(iii) and \cite[Lemma 12.3.1]{DiestelBook10noEE}, any $k$-inseparable set is contained in a unique part of $(\cT_k,\cV_k)$, giving~(i). By \eqref{disting} in the proof of Theorem~\ref{thm_kinsep_sys}, $\cN_k$~distinguishes every two $K$-robust $k$-blocks, which implies (ii) by (i) and Theorem~\ref{treedec}\,(iii).
\end{proof}

From Theorem~\ref{thm_kinsep_treedec} we can finally deduce the two results announced in the Introduction, Theorems 1 and~2.

Theorem~1 follows by taking as $K$ the integer $k$ given in Theorem~1, and then considering the decomposition $(\cT_k,\cV_k)$ for $k=K$. Indeed, consider two $k$-blocks $b_1,b_2$ that Theorem~1 claims are distinguished efficiently by $(\cT_k,\cV_k)$. By Theorem~\ref{thm_kinsep_treedec}\,(ii), $b_1$~and $b_2$ lie in different parts of $(\cT_k,\cV_k)$. Let $k' := \kappa(b_1,b_2)\le k$. By Lemma~\ref{robustii}\,(i),%
   \COMMENT{}
   the $k'$-blocks $b'_1\supseteq b_1$ and $b'_2\supseteq b_2$ are again $K$-robust. Hence by Theorem~\ref{thm_kinsep_treedec}\,(ii) for~$k'$, they lie in different parts of $(\cT_{k'},\cV_{k'})$. Consider an adhesion set of $(\cT_{k'},\cV_{k'})$ on the path in $\cT_{k'}$ between these parts. By Theorem~\ref{thm_kinsep_treedec}\,(iii), this set has size at most~$k'$, and by Theorem~\ref{thm_kinsep_treedec}\,(iv) it is also an adhesion set of $(\cT_k,\cV_k)$ between the two parts of $(\cT_k,\cV_k)$ that contain $b_1$ and~$b_2$.

Theorem~2 follows from Theorem~\ref{thm_kinsep_treedec} for $K=|G|$; recall that robust $k$-blocks are $K$-robust for $K=|G|$.

\section{Outlook}

There are two types of question that arise from the context of this paper, but which we have not addressed.

The first of these concerns its algorithmic aspects. How hard is it

\begin{itemize}
\itemsep0em
\item to decide whether a given graph has a $k$-block;
\item to find all the $k$-blocks in a given graph;
\item to compute the canonical tree-decompositions whose existence we have shown?
\end{itemize}
Note that our definitions leave some leeway in answering the last question. For example, consider a graph $G$ that consists of two disjoint complete graphs $K,K'$ of order~10 joined by a long path~$P$. For $k=5$, this graph has only two $k$-blocks, $K$~and~$K'$. One \td\ of~$G$ that is invariant under its automorphisms has as parts the graphs $K,K'$ and all the $K_2$s along the path~$P$, its decomposition tree again being a long path. This \td\ is particularly nice also in that it also distinguishes the $\ell$-blocks of $G$ not only for $\ell=k$ but for all $\ell$ such that $G$ has an $\ell$-block, in particular, for $\ell=1$.

However if we are only interested in $k$-blocks for $k=5$, this decomposition can be seen as unnecessarily fine in that it has many parts containing no $k$-block. We might, in this case, prefer a \td\ that has only two parts, and clearly there is such a \td\ that is invariant under~$\Aut(G)$, of adhesion 1 or~2 depending on the parity of~$|P|$.

This~\td, however, is suboptimal in yet another respect: we might prefer decompositions in which any part that does contain a $k$-block contains nothing but this $k$-block. Our first decomposition satisfies this, but there is another that does too while having fewer parts: the path-decomposition into three parts whose middle part is~$P$ and whose leaf parts are $K$ and~$K'$.

We shall look at these possibilities and associated algorithms in more detail in~\cite{CDHH13CanonicalAlg, CDHH13CanonicalParts}. However we shall not make an effort to optimize these algorithms from a complexity point of view, so the above three questions will be left open.

Since our \td s are canonical, another obvious question is whether they, or refinements, can be used to tackle the graph isomorphism problem. Are there natural classes of graphs for which we can
\begin{itemize}
\itemsep0em
\item describe the parts of our canonical \td s in more detail;
\item use this to decide graph isomorphism for such classes in polynomial time?
\end{itemize}

Another broad question that we have not touched upon, not algorithmic, is the following. Denote by $\beta(G)$ the greatest integer~$k$ such that $G$ has a $k$-block (or equivalently: has a $k$-inseparable set of vertices). This seems to be an interesting graph invariant; for example, in a network $G$ one might think of the nodes of a $\beta(G)$-block as locations to place some particularly important servers that should still be able to communicate with each other when much of the network has failed.

From a mathematical point of view, it seems interesting to ask how $\beta$ interacts with other graph invariants. For example, what average degree will force a graph to contain a $k$-block for given~$k$? What can we say about the structure of graphs that contain no $k$-block but have large tree-width?

Some preliminary results in this direction are obtained in~\cite{ForcingBlocks}, but even for the questions we address we do not have optimal results.

\section*{Acknowledgement}
Many ideas for this paper have grown out of an in-depth study of the treatise~\cite{DunwoodyKroenArXiv} by Dunwoody and~Kr\"on, which we have found both enjoyable and inspiring.

\small
\bibliographystyle{plain}
\bibliography{collective}

\end{document}